\documentclass[10pt,reqno,draft]{amsart}

\usepackage{mathtools}
\usepackage[T1]{fontenc}

\usepackage{amsmath,amsfonts,amsbsy,amsgen,amscd,mathrsfs,amssymb,amsthm}

\usepackage{enumitem}
\SetLabelAlign{center}{\hfil#1\hfil}

\usepackage{bm}

\usepackage{stmaryrd}
\SetSymbolFont{stmry}{bold}{U}{stmry}{m}{n} 

\usepackage[usenames,dvipsnames]{xcolor}
\usepackage[colorlinks=true,citecolor=blue,linkcolor=blue]{hyperref}

\usepackage{tikz}

\usepackage[font=small,labelfont=bf]{caption}

\newtheorem{thm}{Theorem}[section]
\newtheorem{lem}[thm]{Lemma}

\newtheorem{prop}[thm]{Proposition}
\newtheorem{cor}[thm]{Corollary}

\theoremstyle{definition}

\newtheorem{defn}[thm]{Definition}

\theoremstyle{remark}

\newtheorem{example}[thm]{Example}

\newtheorem{rem}[thm]{Remark}
\newtheorem*{rem*}{Remark}
\newtheorem*{notation*}{Notation}

\numberwithin{equation}{section} 
\numberwithin{figure}{section}
\numberwithin{table}{section}

\newcommand{\Mk}{$(\mathrm{M}_k)$}
\newcommand{\Mkd}{$(\mathrm{M}_k^*)$}
\newcommand{\Ek}{$(\mathrm{E}_k)$}
\newcommand{\Ekd}{$(\mathrm{E}_k^*)$}
\newcommand{\Ck}{$(\mathrm{C}_k)$}
\newcommand{\V}{\mathsf{V}}
\newcommand{\affp}{\mathop{\mathrm{aff}'}}
\newcommand{\sspan}{\mathop{\mathrm{span}}}
\newcommand{\llangle}{\langle\!\langle}
\newcommand{\rrangle}{\rangle\!\rangle}

\begin{document}

\title{The Extremals of the Kahn-Saks Inequality}

\author{Ramon van Handel}
\address{Fine Hall 207, Princeton University, Princeton, NJ 08544, USA}
\email{rvan@math.princeton.edu}

\author{Alan Yan}
\address{Department of Mathematics, Harvard University, 
Cambridge, MA 02138, USA}
\email{alanyan@math.harvard.edu}

\author{Xinmeng Zeng}
\address{ICME, Stanford University, Stanford, CA 94305, USA}
\email{xinzeng@stanford.edu}

\begin{abstract}
A classical result of Kahn and Saks states that given any partially 
ordered set with two distinguished elements, the number of linear 
extensions in which the ranks of the distinguished elements differ by $k$ 
is log-concave as a function of $k$. The log-concave sequences that can 
arise in this manner prove to exhibit a much richer structure, however, 
than is evident from log-concavity alone. The main result of this paper is 
a complete characterization of the extremals of the Kahn-Saks inequality: 
we obtain a detailed combinatorial understanding of where and what kind of 
geometric progressions can appear in these log-concave sequences. This 
settles a partial conjecture of Chan-Pak-Panova, while the analysis 
uncovers new extremals that were not previously conjectured. The proof 
relies on a much more general geometric mechanism---a hard Lefschetz 
theorem for nef classes that was obtained in the setting of convex 
polytopes by Shenfeld and Van Handel---which forms a model for the 
investigation of such structures in other combinatorial problems.
\end{abstract}



\maketitle

\thispagestyle{empty}

\section{Introduction}

A sequence $a_1,\ldots,a_n\ge 0$ is called \emph{log-concave} if
\begin{equation}
\label{eq:logconc}
	a_k^2 \ge a_{k-1}a_{k+1},\qquad\quad k=2,\ldots,n-1. 
\end{equation}
It was observed long ago that many integer sequences that arise in a 
remarkably broad range of combinatorial problems appear to be log-concave 
\cite{Sta89}. Whenever the same mathematical phenomenon arises in many 
different situations, one may wonder whether there is a more fundamental 
underlying mechanism that explains its appearance. The discovery of such 
an explanation---that log-concavity arises due to the presence of 
combinatorial analogues of the Hodge-Riemann relations of algebraic 
geometry---has led to a striking series of recent breakthroughs in the 
understanding of log-concavity in combinatorics \cite{AHK17,Huh18,Huh20}.

While the ubiquity of this mechanism was understood only recently, its 
first appearance dates back to Stanley's inequalities for matroids and 
posets \cite{Sta81}. Stanley studied these problems by expressing the 
combinatorial sequences in question as mixed volumes of convex polytopes, 
whose log-concavity follows from the Alexandrov-Fenchel inequality 
\cite{Sch14}. This classical result of convex geometry, which is a 
far-reaching generalization of the isoperimetric inequality, may also be 
viewed as a special instance of the Hodge-Riemann relations for toric 
varieties (see, e.g., \cite{Huh18} and \cite[\S 5.4]{Ful93}). One 
interpretation of the basic insight behind the recent developments in 
\cite{AHK17,Huh18,Huh20,CP21,CP22} is that while most combinatorial 
problems cannot be reformulated in terms of convex polytopes, one can 
often still prove Alexandrov-Fenchel inequalities directly in the 
combinatorial context.

Even within convex geometry, however, the Alexandrov-Fenchel inequality 
has itself been the subject of an open problem that dates back to 
Alexandrov's original paper \cite{Ale37,Sch85}: when does equality hold? 
This problem is fundamental to the interpretation of the 
Alexandrov-Fenchel inequality as an isoperimetric theorem, as it 
characterizes which bodies are extremal for the geometric quantities that 
appear in the inequality. Recently, this long-standing problem was 
completely resolved in the setting of convex polytopes by Shenfeld and the 
first author \cite{SvH23}. This provides a mechanism for obtaining new 
information on log-concave sequences that arise from the 
Alexandrov-Fenchel inequality: as equality in \eqref{eq:logconc}
$$
	a_k^2 = a_{k-1}a_{k+i} 
	\qquad\Longleftrightarrow\qquad
	\frac{a_{k+1}}{a_k} = \frac{a_k}{a_{k-1}}
$$
corresponds to a geometric progression (when $a_k>0$), the 
characterization of the equality cases provides information on where and 
what kind of geometric progressions can appear in a log-concave sequence. 
The solution of the extremal problem for the Alexandrov-Fenchel inequality 
was exploited in \cite[\S 15]{SvH23} and \cite{MS23} to obtain a detailed 
combinatorial characterization of the geometric progressions that can 
appear in Stanley's poset inequalities \cite{Sta81}, revealing a much 
richer structure in these sequences than is evident from log-concavity 
alone.

The aim of this paper is to develop further insight into such phenomena. 
We investigate a classical log-concavity inequality for posets due to Kahn 
and Saks that plays a central role in their work on sorting problems 
\cite{KS84}. As the proof of this inequality is a direct modification of 
that of Stanley's poset inequalities, one may expect that the 
characterization of its extremals will be similar as well. Surprisingly, 
however, the Kahn-Saks inequality turns out to be much more delicate, and 
its extremals exhibit unexpected new features that are not present in 
Stanley's inequalities. Our results confirm a partial conjecture of 
Chan-Pak-Panova \cite{CPP23}, and uncover further extremals that were not 
previously conjectured. The key challenge in the proofs is to understand 
the geometric features that cause the rich structure of the extremals of 
the Kahn-Saks inequality to appear.

Beside providing a case study of geometric progressions in log-concave 
sequences, our results and those of \cite{SvH23,MS23} may serve as a model 
for the investigation of such phenomena in a broader context. Like the 
Alexandrov-Fechel inequality itself, its equality characterization admits 
an algebraic interpretation as a hard Lefschetz theorem for nef classes 
\cite[\S 16.2]{SvH23}. The development of analogues of this mechanism 
outside convex geometry could provide a common explanation for the 
appearance of geometric progressions in much more general situations (see 
\cite{HX23} for recent progress). At the same time, let us note that 
despite the recent advances in establishing log-concavity in 
combinatorics, no other method appears as of yet to be able to recover the 
Kahn-Saks inequality \cite[\S 7.2]{CP22}. Its rich extremal structure, 
which perches it at the cusp of log-concavity, may help explain why this 
is the case.

\subsection{Main results}

Throughout this paper, we consider an arbitrary partially ordered set 
(poset) $P$ with $n$ elements. Recall that a \emph{linear extension} of 
$P$ is a bijection $f:P\to[n]$ such that $f(z)<f(z')$ whenever $z<z'$.

In the following, we fix two distinguished elements $x,y\in P$ such that 
$x\not\ge y$. For any $k=1,\ldots,n-1$, we denote by $N_k$ the number of 
linear extensions $f$ of $P$ so that $f(y)-f(x)=k$. The fundamental result 
in this setting, due to Kahn and Saks \cite[Theorem 2.5]{KS84}, states 
that the sequence $N_1,\ldots,N_{n-1}$ is log-concave.

\begin{thm}[Kahn-Saks inequality]
\label{thm:ks}
$N_k^2\ge N_{k-1}N_{k+1}$ for $k=2,\ldots,n-2$.
\end{thm}

The aim of our main results is to characterize when equality 
$N_k^2=N_{k-1}N_{k+1}$ holds in the Kahn-Saks inequality.
Before we proceed to their formulation, let us clarify the 
notation that will be used throughout the paper.

\begin{notation*}
Given a clause $C$, we denote by $P_C$ the subset of elements 
of the poset $P$ that satisfy this clause. For example,
\begin{align*}
	P_{\le x} & := \{\omega\in P:\omega\le x\}, \\
	P_{x<\cdot<z} & := \{\omega\in P:x<\omega<z\}, \\
	P_{>x,\parallel y} & := \{\omega\in P:\omega>x\mbox{ and }
	\omega\parallel y\},
\end{align*}
etc. We use the symbol $z\parallel z'$ to denote that $z$ is incomparable 
to $z'$, and write $z\lessdot z'$ to indicate that $z'$ covers $z$ (that 
is, that $z<z'$ and $P_{z<\cdot<z'}=\varnothing$).
\end{notation*}

We begin by observing that equality in Theorem \ref{thm:ks} holds 
trivially when $N_k=0$. The following lemma characterizes when this 
happens (cf.\ \cite[Theorem 8.5]{CPP23}).

\begin{lem}[Vanishing condition]
\label{lem:vanish}
For any $k\in [n-1]$, the following are equivalent.
\begin{enumerate}[leftmargin=*,itemsep=\topsep,label=\alph*.]
\item $N_k=0$.
\item $|P_{<x}|+|P_{>y}|>n-k-1$ or $|P_{x<\cdot<y}|>k-1$.
\end{enumerate}
\end{lem}

The entire difficulty of the problem lies in characterizing the nontrivial 
equality cases, that is, equality in Theorem \ref{thm:ks} with $N_k>0$. To 
this end, we first characterize what kinds of geometric progressions 
can appear. In the following results, we take for granted that
$2\le k\le n-2$ as in Theorem \ref{thm:ks}.

\begin{thm}[Geometric progressions]
\label{thm:maingeom}
If $N_k>0$, the following are equivalent:
\begin{enumerate}[leftmargin=*,itemsep=\topsep,label=\alph*.]
\item $N_k^2 = N_{k-1}N_{k+1}$.
\item Either $N_{k+1}=N_k=N_{k-1}$, or $N_{k+1}=2N_k=4N_{k-1}$.
\end{enumerate}
\end{thm}

That is, only two types of geometric progressions are possible: flat and 
doubling progressions. We will characterize each of these situations 
separately. In order to formulate the equality conditions, we define a 
number of structural properties of the poset $P$ that will appear in 
different combinations below.

\begin{defn}
\label{defn:conditions}
We define the following properties:
\begin{enumerate}[labelwidth=\widthof{\Mkd},itemsep=\topsep,align=center]
\item[\Mk] $|P_{<z}|+|P_{>y}|>n-k$ for all $z\in P_{>x,\not\ge y}$.
\item[\Mkd] $|P_{>z}|+|P_{<x}|>n-k$ for all $z\in P_{<y,\not\le x}$.
\item[\Ek] $|P_{z<\cdot<y}\cup\{x\}|>k$ for all $z\in P_{<x}$,
and $|P_{<y}\cup\{x\}|>k$.
\item[\Ekd] $|P_{x<\cdot<z}\cup\{y\}|>k$ for all $z\in P_{>y}$,
and $|P_{>x}\cup\{y\}|>k$.
\item[\Ck] $|P_{z<\cdot<y}|+|P_{x<\cdot<z'}|>k-2$ for all
$z\in P_{<y,\parallel x}$, $z'\in P_{>x,\parallel y}$ with $z<z'$.
\end{enumerate}
\end{defn}

We first characterize the flat progressions, settling a conjecture of 
Chan-Pak-Panova \cite[Conjecture~8.7]{CPP23} (up to minor corrections, see 
Remark \ref{rem:conditions}).

\begin{thm}[Flat progressions]
\label{thm:mainflat}
If $N_k>0$, the following are equivalent:
\begin{enumerate}[leftmargin=*,itemsep=\topsep,label=\alph*.]
\item $N_{k+1}=N_k=N_{k-1}$.
\item
There is an element $z\in\{x,y\}$ such that for every linear extension $f$ 
of $P$ with $f(y)-f(x)=k$, there exist $u,v\in P_{\parallel z}$
so that $f(u)+1=f(z)=f(v)-1$.
\item Either \Mk{} and \Ek{} hold, or \Mkd{} and \Ekd{} hold.
\end{enumerate}
\end{thm}

In contrast, no plausible conjecture has been put forward on the structure 
of the non-flat equality cases. These are completely settled by the 
following theorem.

\begin{thm}[Doubling progressions]
\label{thm:maindouble}
If $N_k>0$, the following are equivalent:
\begin{enumerate}[leftmargin=*,itemsep=\topsep,label=\alph*.]
\item $N_{k+1}=2N_k=4N_{k-1}$.
\item 
{\binoppenalty=10000 
\relpenalty=10000 
$P=P_{\le x}\cup P_{<y,\parallel x}\cup P_{\ge y}\cup P_{>x,\parallel y}$, 
and for every linear extension $f$ of {$P_{<x}\cup P_{<y,\parallel x}$} 
and every linear extension $f'$ of {$P_{>y}\cup P_{>x,\parallel y}$} the 
following hold: the $k$ largest elements of $f$ are incomparable to $x$, 
the $k$ smallest elements of $f'$ are incomparable to $y$, and for every 
$1\le i\le k-1$ the $i$ largest elements of $f$ are incomparable to the 
$k-i$ smallest elements of $f'$.}\footnote{%
	Here we view a linear extension $f$ of a subset $S\subseteq P$
	as defining a linear ordering of $S$. Thus, e.g., the 
	smallest two elements of $f$ are $f^{-1}(1)$, $f^{-1}(2)$, the 
	largest element of $f$ is $f^{-1}(|S|)$, etc.}
\item 
$P_{\parallel x,\parallel y}=P_{x<\cdot<y}=\varnothing$,
and \Ek{}, \Ekd{} and \Ck{} hold.
\end{enumerate}
\end{thm}

Both Theorems \ref{thm:mainflat} and \ref{thm:maindouble} provide two 
distinct combinatorial characterizations of the extremals of the Kahn-Saks 
inequality. On the one hand, we provide an explicit characterization in 
terms of the structure of the poset itself. This formulation can be 
readily used to verify the equality condition in any concrete situation, 
as the properties in Definition \ref{defn:conditions} can be read off 
directly from the Hasse diagram of $P$. On the other hand, we provide a 
complementary characterization in terms of the structure of the linear 
extensions of $P$. While less explicit, this formulation explains the 
reason that equality arises in each situation (cf.\ section 
\ref{sec:easy}).

\begin{rem}[On Definition \ref{defn:conditions}]
\label{rem:conditions}
Definition \ref{defn:conditions} is formulated for an arbitrary poset $P$ 
and distinguished elements $x\not\ge y$. In most cases, 
\Ek{} and \Ekd{} can be somewhat simplified. For example, in \Ek{}, the 
second condition $|P_{<y}\cup\{x\}|>k$ follows automatically from the 
first condition as long as $P_{<x}\ne\varnothing$; we included it only to 
account for the case that $P_{<x}=\varnothing$. Similarly, when 
$x<y$, we automatically have $x\in P_{z<\cdot<y}$ for any $z\in P_{<x}$ 
and there is no need to include $\{x\}$ separately in \Ek; the present 
formulation accounts also for the case that $x\parallel y$.
These details are essentially cosmetic in nature: as the numbers $N_k$ are 
unchanged if we add a globally minimal and maximal element to $P$ and 
add the relation $x<y$ to the partial order, there would be no loss of 
generality in assuming that $P_{<x},P_{>y}\ne\varnothing$ and $x<y$ in our 
proofs (see, e.g., the proof of Corollary \ref{cor:mainmain}).

On the other hand, other relations among the properties of 
Definition~\ref{defn:conditions} may not be immediately evident. In 
particular, we will show as part of the proof of our main results that the 
conditions \Mk{} and \Ekd{} are mutually exclusive, and that \Mkd{} and 
\Ek{} are mutually exlusive (Lemma \ref{lem:mutex}).
This shows, for example, that the two 
alternatives of Theorem \ref{thm:mainflat}(c) cannot hold simultaneously.

The first four properties of Definition~\ref{defn:conditions} were 
anticipated (up to minor corrections) by Chan-Pak-Panova \cite{CPP23}. 
There, the combination of \Mk{} and \Ek{} was called the \emph{$k$-midway 
property}, and the combination of \Mkd{} and \Ekd{} was called the 
\emph{dual $k$-midway property}. As these properties turn out to appear in 
a different combination in Theorem \ref{thm:maindouble}, we do not
adopt this terminology here.
\end{rem}

\subsection{Discussion}

Our results provide a detailed picture of when, where, and why
geometric progressions can arise in the log-concave sequences 
$N_1,\ldots,N_{n-1}$. The aim of this section is to further discuss the 
significance of these results.

\subsubsection{Shapes of log-concave sequences}

Despite the considerable interest in log-concave sequences of 
combinatorics \cite{Sta89,Huh18}, log-concavity in itself provides a 
limited qualitative picture on what such sequences look like. The study of 
geometric progressions reveals that there can be significant additional 
structure in the shapes of log-concave sequences that arise in a given 
combinatorial problem, and provides more detailed qualitative and 
quantitative information.

In the setting of this paper, the qualitative picture is illustrated in 
Figure \ref{fig:shape}. Recall that any log-concave sequence is unimodal 
on its support. Moreover, by Lemma \ref{lem:vanish}, the sequences in this 
paper can vanish only on an initial or final segment. We therefore claim 
that any extremal sequence (that is, one containing a geometric 
progression) must look qualitatively like one of the plots in Figure 
\ref{fig:shape} (however, some segments $I_j$ may be empty 
in a given situation):
\begin{enumerate}[leftmargin=*,itemsep=\topsep,label=\textbullet]
\item The top plot illustrates the situation with no
doubling progression. By unimodality, there can be at most one
flat segment that is the maximum of the sequence.
\item The bottom plot 
illustrates the situation where there is a doubling progression. 
Note that the conditions \Ek, \Ekd, \Ck\ automatically imply 
$(\mathrm{E}_l)$, $(\mathrm{E}_l^*)$, $(\mathrm{C}_l)$ for all $l<k$.
Thus Theorem \ref{thm:maindouble} and Lemma \ref{lem:vanish} imply that 
a doubling progression can only arise as the initial segment of the sequence.
\end{enumerate}
Beyond this qualitative picture, however, our results provide detailed 
quantitative information: they enable us to compute precisely where in the 
sequence the geometric progressions appear (that is, we can compute the
length of each segment $I_j$ explicitly in terms of the poset 
structure). It appears rather surprising that such detailed information is 
accessible for the linear extension numbers of arbitrary posets, which are 
themselves hard to compute \cite{BW91}.
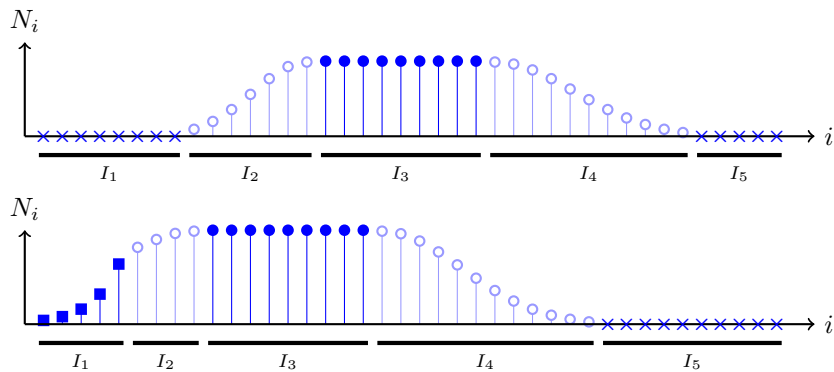
\begin{figure}
\centering
\captionsetup{width=.925\linewidth}

\begin{tikzpicture}[scale=.25]
\newcommand\Square[1]{+(-#1,-#1) rectangle +(#1,#1)}

\foreach \i in {9,...,15}
{
        \draw[color=blue!40!white] (\i,{-.05+4*exp(-(\i-15)^2/16)}) -- 
(\i,0);
        \draw[color=blue!40!white,fill=white,thick] 
(\i,{-.05+4*exp(-(\i-15)^2/16)}) circle (0.25);
}

\foreach \i in {16,...,24}
{
        \draw[color=blue,fill=blue,thick] (\i,4) circle (0.25);
        \draw[color=blue] (\i,4) -- (\i,0);
}

\foreach \i in {25,...,35}
{
        \draw[color=blue!40!white] (\i,{-.05+4*exp(-(25-\i)^2/36)}) -- 
(\i,0);
        \draw[color=blue!40!white,fill=white,thick] 
(\i,{-.05+4*exp(-(25-\i)^2/36)}) circle (0.25);
}

\draw[thick,->] (0,0) -- (0,5) node[above] {$N_i$};
\draw[thick,->] (0,0) -- (42,0) node[right] {$i$};

\draw[ultra thick] (.75,-1) -- (8.25,-1) node[midway,below] 
{$\scriptstyle I_1$};
\draw[ultra thick] (8.75,-1) -- (15.25,-1) node[midway,below] 
{$\scriptstyle I_2$};
\draw[ultra thick] (15.75,-1) -- (24.25,-1) node[midway,below] 
{$\scriptstyle I_3$};
\draw[ultra thick] (24.75,-1) -- (35.25,-1) node[midway,below] 
{$\scriptstyle I_4$};
\draw[ultra thick] (35.75,-1) -- (40.25,-1) node[midway,below] 
{$\scriptstyle I_5$};

\foreach \i in {1,...,8}
{
        \draw[color=blue,fill=blue,thick] (\i,0) node 
{$\scriptstyle\boldsymbol{\times}$};
}

\foreach \i in {36,...,40}
{
        \draw[color=blue,fill=blue,thick] (\i,0) node
{$\scriptstyle\boldsymbol{\times}$};
}

\begin{scope}[yshift=-10cm]

\foreach \i in {1,...,5}
{
        \draw[color=blue,fill=blue,thick] (\i,{(2^\i)/10}) \Square{0.25};
	\draw[color=blue] (\i,{(2^\i)/10}) -- (\i,0);
}

\foreach \i in {6,...,9}
{
        \draw[color=blue!40!white] (\i,{2.95+2*exp(-(\i-9)^2/16)}) -- 
(\i,0);
        \draw[color=blue!40!white,fill=white,thick] 
(\i,{2.95+2*exp(-(\i-9)^2/16)}) circle (0.25);
}

\foreach \i in {10,...,18}
{
        \draw[color=blue,fill=blue,thick] (\i,5) circle (0.25);
        \draw[color=blue] (\i,5) -- (\i,0);
}

\foreach \i in {19,...,30}
{
        \draw[color=blue!40!white] (\i,{-.05+5*exp(-(19-\i)^2/36)}) -- 
(\i,0);
        \draw[color=blue!40!white,fill=white,thick] 
(\i,{-.05+5*exp(-(19-\i)^2/36)}) circle (0.25);
}

\draw[thick,->] (0,0) -- (0,5) node[above] {$N_i$};
\draw[thick,->] (0,0) -- (42,0) node[right] {$i$};

\draw[ultra thick] (.75,-1) -- (5.25,-1) node[midway,below] 
{$\scriptstyle I_1$};
\draw[ultra thick] (5.75,-1) -- (9.25,-1) node[midway,below] 
{$\scriptstyle I_2$};
\draw[ultra thick] (9.75,-1) -- (18.25,-1) node[midway,below] 
{$\scriptstyle I_3$};
\draw[ultra thick] (18.75,-1) -- (30.25,-1) node[midway,below] 
{$\scriptstyle I_4$};
\draw[ultra thick] (30.75,-1) -- (40.25,-1) node[midway,below] 
{$\scriptstyle I_5$};

\foreach \i in {31,...,40}
{
        \draw[color=blue,fill=blue,thick] (\i,0) node
{$\scriptstyle\boldsymbol{\times}$};
}

\end{scope}

\end{tikzpicture}
\vspace{-.2cm}

\caption{Structure of extremal sequences.
The symbols denote vanishing 
(\raisebox{1pt}{$\scriptstyle\boldsymbol{\times}$}),
flat progressions ($\bullet$), doubling progressions 
(\raisebox{0.5pt}{$\scriptscriptstyle\blacksquare$}),
and strictly log-concave regions ($\circ$).
\label{fig:shape}}
\end{figure}

It is not immediately obvious from the formulation of our results 
that all regions of the log-concave sequences that are illustrated in 
Figure \ref{fig:shape} can in fact appear simultaneously. This is however 
readily verified by means of simple examples.

\begin{example}
Suppose that $x$ is a globally minimal element of $P$, that is, $x\le z$ 
for all $z\in P$. Then $x$ must appear first in any linear extension of 
$P$. Therefore, as noted in \cite{CPP23}, this special case reduces to the 
situation originally studied by Stanley: $N_k$ is the number of linear 
extensions of $P\backslash\{x\}$ in which $y$ has rank $k$. In this 
setting, the explicit construction of \cite[Example 15.4]{SvH23} shows 
that we may engineer the poset $P$ to achieve a log-concave sequence as in 
the top plot of Figure \ref{fig:shape} with an essentially arbitrary 
choice of (positive) lengths of the segments $I_1,\ldots,I_5$.
\end{example}

\begin{example}
\label{ex:doublefull}
We now provide a counterpart of the previous example for the bottom plot 
of Figure \ref{fig:shape}. Consider the poset $P$ with $|P|=r+s+2$
defined by the relations
$$
	x\lessdot z_1 \lessdot \cdots \lessdot z_r,\qquad
	w_1\lessdot \cdots \lessdot w_s,\qquad
	w_u \lessdot z_v,\qquad
	w_u \lessdot y\lessdot z_t
$$
for any $v<t\le r$ and $u\le s$.
Then $P_{x<\cdot<y}=P_{\parallel x,\parallel y}=\varnothing$,
and $N_k>0$ for $k\le t+s$ by Lemma \ref{lem:vanish}.
A straightforward computation shows that
\Mk\ holds when $k>t+s$,
\Mkd\ holds when $k>u+v$, 
\Ek\ holds when $k\le u$,
\Ekd\ holds when $k<t$, and
\Ck\ holds when $k\le v$. Thus Theorems 
\ref{thm:mainflat} and \ref{thm:maindouble}
yield 
\begin{alignat*}{3}
	&N_{k+1}=2N_k=4N_{k-1} && \quad\text{if and only if}\quad 2\le 
k\le 
\min(u,v),\\
	&N_{k+1}=N_k=N_{k-1} && \quad\text{if and only if}\quad u+v<k<t,\\
	&N_k=0 && \quad\text{if and only if}\quad t+s<k\le |P|.
\end{alignat*}
In particular, we obtain a log-concave sequence as in the bottom plot 
of Figure~\ref{fig:shape}, and the lengths of the segments 
$I_1,\ldots,I_5$ can be chosen essentially arbitrarily by selecting the 
parameters $r,s,t,u,v$ appropriately.
\end{example}

The poset in Example \ref{ex:doublefull} was constructed in order to 
illustrate the bottom plot of Figure \ref{fig:shape} in full generality. A 
much simpler example where a doubling progression arises is discussed in 
section \ref{sec:weird} (see also \cite[\S 1.4]{CPP23}).

Additional qualitative features of the log-concave sequences of Kahn and 
Saks can be deduced from our main results. As an example, we include the 
following curious observation that was suggested to us by a referee.

\begin{example}
Suppose that $x\parallel y$. Then, by a slight abuse of notation, we can 
define $N_{-1},\ldots,N_{-(n-1)}$ where $N_{-k}$ denotes 
the number of linear extensions $f$ of $P$ so that $f(x)-f(y)=k$.
This sequence is again log-concave by Theorem \ref{thm:ks}, as it 
corresponds to the sequence $N_1,\ldots,N_{n-1}$ associated to the dual 
poset $P^*$ defined by reversing the partial order of $P$.
(We emphasize that negative indices, which appear in \cite{KS84}, are not 
used elsewhere in this paper.)

It is certainly possible that either the sequence $N_1,\ldots,N_{n-1}$ or 
$N_{-1},\ldots,N_{-(n-1)}$ contains a doubling progression. However, it is 
not possible that both these sequences contain a doubling progression 
simultaneously for the same poset $P$.

Indeed, suppose that both sequences contain a doubling progression. Then 
Theorem~\ref{thm:maindouble} implies not only that \Ek{} must hold for 
$k=2$, but that this condition must also hold with the roles of $x$ and 
$y$ reversed (the latter is the condition \Ekd{} associated to the dual 
poset $P^*$). As $P_{<y}\ne\varnothing$ by \Ek{}, we may choose $z\lessdot 
y$. Then \Ek{} with $x$ and $y$ reversed implies that $z<x$. Thus \Ek{} 
implies that $P_{z<\cdot<y}\ne\varnothing$. This entails a contradiction, 
as we assumed $z\lessdot y$.
\end{example}

\subsubsection{General mechanisms}

One of the main reasons for the interest of the log-concavity
phenomenon in combinatorics is that it suggests the presence of a certain 
universal 
structure\footnote{%
        In fact, the Kahn-Saks inequality is somewhat unusual in that
        it was developed in \cite{KS84} as a tool to investigate
        sorting problems, rather than being motivated by
        log-concavity in itself.}
that appears to arise in many combinatorial problems: combinatorial 
analogues of Alexandrov-Fenchel inequalities \cite{Huh20,CP21,CP22} or of 
more general Hodge-Riemann relations \cite{AHK17,Huh18}. It is remarkable 
that the same kinds of structures are fundamental to several other areas 
of mathematics, including convex geometry \cite{Sch14}, algebraic geometry 
\cite{Tei82,Ful93}, and complex geometry \cite{Gro90}.

In contrast to log-concavity, which is a robust qualitative property, one 
might expect that the understanding of geometric progressions in 
log-concave sequences must be developed on a case-by-case basis. We 
believe, however, that geometric progressions likely arise in many 
problems through a common mechanism that is nearly as universal as the 
structure that gives rise to log-concavity itself. As is explained in 
\cite[\S 16.2]{SvH23}, the equality cases of the Alexandrov-Fenchel 
inequality that form the basis for the analysis in this paper admit a 
natural algebraic interpretation (a precise counterpart of the hard 
Lefschetz theorem of degree $1$ for nef classes). In this form, such 
structures can be meaningfully formulated in much more general situations, 
and it is natural to conjecture that they do indeed arise in many 
combinatorial problems. Such questions are largely open to date, though 
there is recent progress on analogous questions in algebraic geometry 
\cite{HX23}.

The results of this paper and those of \cite{SvH23,MS23} motivate the 
investigation of such structures, and illustrate the kind of strong 
qualitative and quantitative information that can be obtained when they 
are present.

\subsubsection{Interplay between geometry and combinatorics}

The proof of the Kahn-Saks inequality, like that of Stanley's inequality 
on which it is based, translates the original combinatorial problem to a 
geometric problem for convex polytopes. As will be explained in section 
\ref{sec:af}, Kahn and Saks construct two $(n-1)$-dimensional convex 
polytopes $K,L$ so that the linear extension numbers can be represented as 
a mixed volume $N_k = (n-1)!\,\V(K[n-k],L[k-1])$. The inequality 
$N_k^2\ge N_{k-1}N_{k+1}$ then 
follows immediately from the Alexandrov-Fenchel inequality
$$
	\V(K,L,C_1,\ldots,C_{n-3})^2 \ge
	\V(K,K,C_1,\ldots,C_{n-3})\,
	\V(L,L,C_1,\ldots,C_{n-3}),
$$
which is a far-reaching generalization of the isoperimetric inequality 
in convex geometry. (We refer to \cite{Sch14} for background on convex 
geometry; the relevant notions for this paper will be briefly reviewed in 
section \ref{sec:af}.) One might view the Kahn-Saks inequality 
as an isoperimetric inequality for posets, and its equality 
characterization as the corresponding isoperimetric theorem.

The equality cases of the Alexandrov-Fenchel inequality were completely 
characterized for convex polytopes in \cite{SvH23}. However, this 
characterization provides geometric information on the polytopes in 
question. The main difficulty in the proof of our main results is to 
understand how to translate this geometric information back to the 
original combinatorial problem. This translation is far from 
straightforward, and requires us to develop a detailed understanding 
of the interplay between the geometric and combinatorial descriptions.

An unexpected consequence of the results of this paper is that they 
complete the picture of what geometric extremals can arise in 
combinatorial problems. Let us recall from \cite[\S 2.2]{SvH23} that the 
equality cases of the Alexandrov-Fenchel inequality arise from a 
superposition of three distinct mechanisms: 
\begin{enumerate}[leftmargin=*,itemsep=\topsep,label=\textbf{M\arabic*}.] 
\item 
Translation and scaling. \item Degeneration of the support of mixed area 
measures. 
\item ``Critical'' equality cases caused by dimensional collapse. 
\end{enumerate}
As there is enormous freedom in the choice of arbitrary convex polytopes, 
it may not be too surprising that the Alexandrov-Fenchel inequality has 
numerous equality cases. What is far more surprising is that \emph{all} 
the equality mechanisms turn out to arise even in natural combinatorial 
applications. While only \textbf{M2} plays a role in the setting of 
Stanley's inequality \cite[\S 15]{SvH23}, \textbf{M3} can arise in a 
general form of Stanley's inequality \cite{MS23} (the latter has striking 
complexity implications \cite{CP24}). It will turn out that \textbf{M1} 
plays a central role in this paper, completing the picture.

As the role of \textbf{M1} is a novel feature of the problem investigated 
in this paper that did not arise in previous works, we briefly discuss it 
further.

\subsubsection{The role of translation and scaling}
\label{sec:transscale}

Geometrically, translation and scaling may be viewed as trivial reasons 
for equality. For example, equality clearly holds in the 
Alexandrov-Fenchel inequality when $K=L$; therefore, as mixed volumes are 
homogeneous and translation invariant, we trivially obtain new equality 
cases $K=aL+v$ by translation and scaling. Nontrivial equality cases arise 
due to \textbf{M2} above, which states that these polytopes need not be 
equal but need only have the same supporting hyperplanes in some 
directions (cf.\ section \ref{sec:eqaf}).

From a combinatorial perspective, however, the appearance of nontrivial 
translation and scaling is unexpected. The connection between convexity 
and combinatorics arises when we work with lattice polytopes, in which 
case mixed volumes are always combinatorial quantities \cite[\S 
5.4]{Ful93}. In particular, the polytopes used by Stanley, Kahn and Saks 
have all their vertices in the set $\{0,1\}^n$, cf.\ \cite{Sta86}. One may 
expect that this property provides sufficient rigidity that two such 
polytopes cannot have the same supporting hyperplanes if we scale or 
translate one of them. It is a surprising feature of the Kahn-Saks 
inequality that nontrivial scaling and translation can nonetheless arise. 
In particular, the resulting equality cases of the Alexandrov-Fenchel 
inequality do not respect the lattice structure of the underlying 
polytopes, as will be illustrated in a simple example in section 
\ref{sec:weird}.

The consequences of this breakdown of rigidity are fundamental to our main 
results. Nontrivial scaling is responsible for the presence of the 
non-flat geometric progressions (cf.\ section \ref{sec:eqaf}). At the same 
time, nontrivial translations arise in the present setting even when we 
consider flat geometric progressions: the translation determines which of 
the two cases of Theorem \ref{thm:mainflat}(c) is in force, as will become 
clear in the proof (cf.\ Definition \ref{defn:condpf} and the following 
discussion).

\subsection{Organization of this paper}
\label{sec:org}

The rest of this paper is organized as follows.

The implications (c)$\Rightarrow$(b)$\Rightarrow$(a) of Theorems 
\ref{thm:mainflat} and \ref{thm:maindouble} require only elementary 
arguments. We first dispense with these implications in section 
\ref{sec:easy}. The remainder of the paper is devoted to the core part of 
our main results---Theorem \ref{thm:maingeom} and the implication 
(a)$\Rightarrow$(c) of Theorems \ref{thm:mainflat} and 
\ref{thm:maindouble}.

In section \ref{sec:af}, we describe the convex geometric construction 
that underpins the Kahn-Saks inequality. This yields, through the equality 
cases of the Alexandrov-Fenchel inequality, a geometric description of its 
extremals. The main difficulty in the proof is now to understand how one 
can use this geometric information to characterize the combinatorial 
structure of the poset.

In section \ref{sec:face}, we aim to translate the basic geometric data 
that was obtained in section \ref{sec:af} into combinatorial conditions. 
This requires us to study the facial structure of the polytopes that 
appear in the geometric equality characterization. The faces will turn out 
to be described by combinatorial constraints. As a byproduct, we 
also obtain a short proof of Lemma \ref{lem:vanish} here.

Now that the geometric data has been converted to combinatorial 
information, we arrive at the heart of the argument: we must use this 
combinatorial information to fully characterize the structure of the 
poset. This is the main part of the proof of our main results, which is 
contained in section \ref{sec:proof}.

\section{Sufficiency}
\label{sec:easy}

The main results of this paper provide necessary and sufficient conditions 
for equality to hold in the Kahn-Saks inequality. One direction of these 
results is considerably simpler than the other, however: the proof that 
the combinatorial conditions of Theorems \ref{thm:mainflat} and 
\ref{thm:maindouble} are sufficient for equality to hold is entirely 
elementary. Nonetheless, this direction sheds considerable light on the 
structure of the problem, as its proof reveals the combinatorial 
mechanisms that give rise to equality. This direction of our main results 
will be proved in the present section. The main challenge ahead of us is 
to show that these are the \emph{only} mechanisms that can give rise to 
equality, that is, that the sufficient conditions are also necessary. The 
proof of the latter will occupy the remainder of this paper.

We begin by proving the implications (c)$\Rightarrow$(b)$\Rightarrow$(a) 
of Theorem \ref{thm:mainflat}. The argument is essentially contained in 
\cite[\S 8.3]{CPP23}, and is very similar to the argument in the proof of 
\cite[Theorem 15.3]{SvH23}. We include the proof for completeness.

\begin{proof}[Proof of Theorem \ref{thm:mainflat}: 
\emph{(c)$\Rightarrow$(b)}]
Suppose that \Mk{} and \Ek{} hold. Let $f$ be any linear extension of $P$ 
such that $f(y)-f(x)=k$. Since $|P_{<y}\cup\{x\}|>k$ by \Ek{}, we 
must have $f(y)\ge k+2$ and thus $f(x)=f(y)-k\ge 2$.

Now consider $u\in P$ such that $f(u)=f(x)-1$, which exists as $f(x)\ge 2$.
If it were the case that $u<x$, then \Ek{} would imply that
$$
	k<|P_{u<\cdot<y}\cup\{x\}|\le
	f(y)-f(u)-1=k,
$$
which is impossible. But $u>x$ is also impossible as $f(u)<f(x)$. Thus 
$u\parallel x$. 

Similarly, consider $v\in P$ such that $f(v)=f(x)+1$. 
Note that $v\not\ge y$ because $f(v)<f(x)+2\le f(y)$ (as $k\ge 2$).
If $v>x$, then \Mk{} would imply that
$$
	n-k < |P_{<v}|+|P_{>y}| \le f(v)-1 + n-f(y) = n-k,
$$
which entails a contradiction. We conclude as above that $v\parallel x$.

We have therefore shown that the validity of \Mk{} and \Ek{} implies
that condition (b) of Theorem \ref{thm:mainflat} holds with $z=x$. We omit 
the completely analogous argument that \Mkd{} and \Ekd{} imply condition (b) 
of Theorem \ref{thm:mainflat} with $z=y$.
\end{proof}

\begin{proof}[Proof of Theorem \ref{thm:mainflat}: 
\emph{(b)$\Rightarrow$(a)}]
Suppose that condition (b) of Theorem \ref{thm:mainflat} holds with
$z=x$. 
Let $f$ be any linear extension of $P$
such that $f(y)-f(x)=k$ and $f(u)+1=f(x)=f(v)-1$. As $u,v\parallel x$,
swapping $v$ and $x$ in the linear order defined by $f$ yields a new 
linear extension $f'$ of $P$ with $f'(y)-f'(x)=k-1$, while swapping $u$ 
and $x$ yields a linear extension $f''$ of $P$ with 
$f''(y)-f''(x)=k+1$.

As the maps $f\mapsto f'$ and $f\mapsto f''$ are clearly injective, it 
follows that $N_k\le N_{k-1}$ and $N_k\le N_{k+1}$.
Thus the Kahn-Saks inequality yields
\begin{align*}
	&N_{k-1} N_k
	\ge
	N_k^2 \ge N_{k-1}N_{k+1} \ge N_{k-1} N_k,\\
	&N_k N_{k+1}
	\ge
	N_k^2 \ge N_{k-1}N_{k+1} \ge N_k N_{k+1}.
\end{align*}
Since we assume that $N_k>0$, it follows that
$N_{k-1}=N_k=N_{k+1}$.

We have therefore shown that condition (b) of Theorem \ref{thm:mainflat} 
with $z=x$ implies condition (a). We omit the completely analogous 
argument in the case $z=y$.
\end{proof}

We now turn to the proof of the implications 
(c)$\Rightarrow$(b)$\Rightarrow$(a) of Theorem \ref{thm:maindouble}. The 
mechanism by which equality arises here is very different than in Theorem 
\ref{thm:mainflat}: to prove it, we will establish a bijection between 
certain sets of linear extensions.

\begin{proof}[Proof of Theorem \ref{thm:maindouble}: 
\emph{(c)$\Rightarrow$(b)}]
That $P=P_{\le x}\cup P_{<y,\parallel x}\cup P_{\ge y}\cup P_{>x,\parallel y}$
is merely a reformulation of the conditions
$P_{\parallel x,\parallel y}=P_{x<\cdot<y}=\varnothing$. In the rest 
of the proof, we fix any linear extension $f$ of $P_{<x}\cup 
P_{<y,\parallel x}$ and $f'$ of $P_{>y}\cup P_{>x,\parallel y}$.

Let $z\in P_{>y}$. As $P_{x<\cdot<y}=\varnothing$, we have
$P_{x<\cdot<z}\backslash\{y\}\subseteq P_{>y}\cup P_{>x,\parallel y}$.
Thus
$$
	k < |P_{x<\cdot<z}\cup\{y\}| =
	|P_{x<\cdot<z}\backslash\{y\}|+1
	\le f'(z)
$$
by \Ekd, that is, we have shown that every $z\in P_{>y}$ has
rank at least $k+1$ with respect to $f'$. On the other hand, if
$P_{>y}=\varnothing$ then $|P_{>x,\parallel 
y}|=|P_{>x}\backslash\{y\}|\ge k$ by \Ekd. In either case,
we conclude that the 
$k$ smallest elements of $f'$ must lie in $P_{>x,\parallel y}$.
The completely analogous argument using \Ek{} instead of \Ekd{} shows that 
the $k$ largest elements of $f$ must lie in $P_{<y,\parallel x}$.

Now fix $1\le i\le k-1$, let $z\in P_{<y,\parallel x}$ be among the $i$ 
largest elements of $f$, and let $z'\in P_{>x,\parallel y}$ be among
the $k-i$ smallest elements of $f'$. Note that we must have
$P_{z<\cdot<y}\subseteq P_{<y,\parallel x}$ and
$P_{x<\cdot<z'}\subseteq P_{>x,\parallel y}$ as
$P_{x<\cdot<y}=\varnothing$, $z\parallel x$, and
$z'\parallel y$, so that
$$
	|P_{z<\cdot<y}|<i,\qquad\quad
	|P_{x<\cdot<z'}|<k-i.
$$
Thus \Ck{} implies that $z\not<z'$. On the other hand, $z\not\ge z'$ as 
$P_{x<\cdot<y}=\varnothing$, so we must have $z\parallel z'$. We have 
therefore shown that the $i$ largest elements of $f$ are incomparable to 
the $k-i$ smallest elements of $f'$, concluding the proof.
\end{proof}

\begin{proof}[Proof of Theorem \ref{thm:maindouble}: 
\emph{(b)$\Rightarrow$(a)}]
Let $P_- := P_{<x}\cup P_{<y,\parallel x}$ and
$P_+ := P_{>y}\cup P_{>x,\parallel y}$.
Note that $P_-,P_+,\{x,y\}$ are disjoint as
$x\not\ge y$, and that $P=P_-\cup P_+\cup \{x,y\}$.

Let $1\le \ell\le k+1$. Denote by $\mathcal{E}_-,\mathcal{E}_+$ the sets 
of all linear extensions of $P_-,P_+$, and by $\mathcal{E}_\ell$ the set 
of all linear extensions $\bar f$ of $P$ such that $\bar f(y)-\bar 
f(x)=\ell$. We aim to construct a bijection $\iota:\mathcal{E}_\ell\to 
\mathcal{E}_-\times\mathcal{E}_+\times \{0,1\}^{\ell-1}$.

The construction is as follows. Given any $\bar f\in\mathcal{E}_\ell$, let 
$f\in\mathcal{E}_-$ and $f'\in\mathcal{E}_+$ be defined by restricting the 
linear order of $\bar f$ to $P_-$ and $P_+$, respectively. Moreover, 
define $\omega_i=0$ if the $i$th smallest element of $\bar f$ between $x$ 
and $y$ is in $P_-$, and $\omega_i=1$ if the $i$th smallest element of 
$\bar f$ between $x$ and $y$ is in $P_+$. This defines a map $\iota:\bar 
f\mapsto (f,f',\omega)$. We must show this map is injective and 
surjective.

To show $\iota$ is injective, note that by the definitions of 
$P_-,P_+,\mathcal{E}_\ell$, every element of $P_-$ must be smaller than 
$y$ and every element of $P_+$ must be larger than $x$ in the linear 
ordering defined by $\bar f$. Thus $\bar f$ can be uniquely reconstructed 
from $f,f',\omega$ by choosing the elements between $x$ and $y$ in order 
from the largest elements of $f$ and the smallest elements of $f'$ as 
defined by $\omega$, and placing the remaining elements of $f$ and $f'$ 
below $x$ and above $y$, respectively. This proves injectivity of $\iota$.

Note that the above reconstruction procedure enables us to define a linear 
ordering $\bar f$ of $P$ with $\bar f(y)-\bar f(x)=\ell$ starting from 
\emph{any} $(f,f',\omega)\in \mathcal{E}_-\times\mathcal{E}_+\times 
\{0,1\}^{\ell-1}$. To prove $\iota$ is surjective, we must show that any 
linear ordering $\bar f$ thus defined is a linear extension of $P$, that 
is, that it is compatible with the partial order of $P$. In other words, 
we must show that $\bar f(u)<\bar f(v)$ implies $u\not\ge v$.
\begin{enumerate}[leftmargin=*,itemsep=\topsep,label=\textbullet]
\item For $u,v\in P_-\cup\{y\}$ or $u,v\in P_+\cup\{x\}$, this follows
by construction as $f,f'$ are linear extensions of $P_-\subseteq 
P_{\not\ge y}$ and $P_+\subseteq P_{\not\le x}$, respectively.
\item For $u\in P_-\cup\{y\}$ and $v\in P_+\cup\{x\}$, we have
$u\not\ge v$ by the definitions of $P_-,P_+$ and as $x\not\ge y$, 
$P_{x<\cdot<y}=\varnothing$ (except $u=y,v=x$ for which
$\bar f(u)\not<\bar f(v)$).
\item 
For $u\in P_+\cup\{x\}$ and $v\in P_-\cup\{y\}$, the construction of $\bar f$
ensures that if $u\in P_+$ it must be among the $j$ smallest elements 
of $f'$, and if $v\in P_-$ it must be among the $\ell-1-j$ largest 
elements of $f$, where $j=\sum_{i=1}^{\ell-1}\omega_i$. Thus condition (b) of 
Theorem \ref{thm:maindouble} yields $u\parallel v$, unless $u=x,v=y$ in which case 
$u\not\ge v$ by assumption.
\end{enumerate}
Thus $\bar f$ is a linear extension of $P$, proving surjectivity of $\iota$.

As we have now shown that $\iota$ is a bijection, it follows that
$$
	N_\ell = |\mathcal{E}_\ell| = 2^{\ell-1}{|\mathcal{E}_-|}{|\mathcal{E}_+|}
$$
for every $1\le \ell\le k+1$. In particular,
$N_{k+1}=2N_k=4N_{k-1}$.
\end{proof}

\section{Geometric description of the extremals}
\label{sec:af}

The aim of this section is to recall the geometric construction that 
gives rise to the Kahn-Saks inequality \cite{KS84}, and to deduce a 
geometric characterization of its equality cases from the extremals of the 
Alexandrov-Fenchel inequality \cite{SvH23}. This geometric description of 
the equality cases forms the basis for the remainder of the paper: the 
main challenge in the following sections will be to understand how to use 
this geometric information to extract combinatorial structure.

Before we proceed, let us recall some basic notions of convex geometry 
that will be used without comment in the following. Given two convex 
bodies $C,C'$,
$$
	C+C' := \{x+y:x\in C,y\in C'\}
$$
denotes their Minkowski sum. The support function $h_C$ of $C$ is 
defined by
$$
	h_C(u) := \sup_{x\in C} \langle u,x\rangle.
$$
If $\|u\|=1$, then $h_C(u)$ may be interpreted as the signed distance from 
the origin to the supporting hyperplane of $C$ with outer normal $u$. The 
set
$$
	F(C,u) := \{x\in C: \langle u,x\rangle = h_C(u)\}
$$
is the (exposed) face of $C$ with outer normal $u$. Finally, we recall 
that
$h_{C+C'}(u)=h_C(u)+h_{C'}(u)$ and $F(C+C',u)=F(C,u)+F(C',u)$ 
\cite[Theorem 1.7.5]{Sch14}.

\subsection{The geometric construction}

The following setting will be used throughout the 
rest of this paper. Let $\mathbb{R}^P$ be the $n$-dimensional real vector 
space that is spanned by the coordinate basis $\{e_z\}_{z\in P}$. We 
always equip $\mathbb{R}^P$ with the standard inner product that makes the 
coordinate basis orthonormal. For any $t\in\mathbb{R}^P$, we denote 
its coordinates as $t_z := \langle e_z,t\rangle$. Finally, we define
$$
	V := \{e_y-e_x\}^\perp =
	\{t\in\mathbb{R}^P:t_y-t_x=0\},
$$
where $x\not\ge y$ are the distinguished elements of $P$.

A fundamental role in the following will be played by the order polytope 
of $P$, and in particular by two special slices of the order polytope.

\begin{defn}
The order polytope of $P$ is defined as
$$
	O_P := \{t\in[0,1]^P:t_z\le t_{z'}\text{ whenever }z<z'\}.
$$
Moreover, the polytopes
\begin{align*}
	K & :=
	\{ t\in O_P : t_y-t_x=0 \}, \\
	L & := 
	\{ t\in O_P : t_y-t_x=1\}
\end{align*}
will be fixed throughout the paper.
\end{defn}

As $K\subset V$ and $L\subset V+e_y$, every Minkowski combination 
$(1-\lambda)K+\lambda L$ lies in a translate of the $(n-1)$-dimensional 
space $V$. By a classical fact of Minkowski \cite[\S 5.1]{Sch14},
the $(n-1)$-dimensional volume then satisfies 
$$
	\mathrm{V}_{n-1}((1-\lambda)K+\lambda L) =
	\sum_{k=1}^n \binom{n-1}{k-1}
	(1-\lambda)^{n-k}\lambda^{k-1}\,
	\V(K[n-k],L[k-1]),
$$
where $\V(K[n-k],L[k-1])$ denotes the $(n-1)$-dimensional \emph{mixed 
volume} in which $K$ appears $n-k$ times and $L$ appears 
$k-1$ times. These mixed volumes turn out to compute the linear 
extension numbers $N_k$ \cite[eq.\ (2.14)]{KS84}.

\begin{lem}[Kahn-Saks]
\label{lem:mvol}
$N_k = (n-1)!\,\V(K[n-k],L[k-1])$.
\end{lem}

With Lemma \ref{lem:mvol} in hand, the Kahn-Saks inequality 
$N_k^2\ge N_{k-1}N_{k+1}$ follows immediately from the
\emph{Alexandrov-Fenchel inequality} \cite[Theorem 7.3.1]{Sch14}
\begin{multline*}
	\V(K,L,K[n-k-1],L[k-2])^2 \ge \\
	\V(K,K,K[n-k-1],L[k-2])\,
	\V(L,L,K[n-k-1],L[k-2]).
\end{multline*}
Equality in the Kahn-Saks inequality is therefore also
a special case of equality in the
Alexandrov-Fenchel inequality, which is characterized in \cite{SvH23}. The 
latter will provide an explicit description of the equality conditions in 
terms of the geometry of the polytopes $K$ and $L$, which we turn to next.

\begin{rem}
\label{rem:dontworry}
In this paper, mixed volumes will only be used in order to apply the 
results of \cite{SvH23} and do not appear outside this section. For this 
reason, we have omitted most background material on mixed volumes, 
referring the interested reader to the excellent monograph \cite{Sch14} 
for a detailed treatment.

Let us note that results on mixed volumes are usually formulated for 
convex bodies lying in a given $n$-dimensional space, whereas the 
polytopes $K$ and $L$ lie in different translates of the 
$(n-1)$-dimensional space $V$. We can readily reduce the latter setting to 
the former (with $n\leftarrow n-1$) by replacing $L$ by 
$L'=L-\frac{e_y-e_x}{2}\subset V$. By translation-invariance of mixed 
volumes and as $h_L(u)=h_{L'}(u)$ for all $u\in V$, any resulting 
statements for $K,L'\subset V$ hold verbatim for $K,L$.
\end{rem}

\subsection{Equality cases}
\label{sec:eqaf}

At the heart of the extremal characterization of the Kahn-Saks inequality 
lies the notion of a $k$-extreme vector.

\begin{defn}
\label{defn:kextreme}
A vector $u\in V$ is said to be \emph{$k$-extreme} if the following hold:
\begin{align*}
	\dim F(K,u) &\ge n-k-1, \\
	\dim F(L,u) &\ge k-2,\\
	\dim F(K+L,u) &\ge n-3.
\end{align*}
\end{defn}

The following is the main result of this section.

\begin{prop}[Kahn-Saks extremals: geometric characterization]
\label{prop:geom}
Let $a>0$, and assume that $N_k>0$. Then the following are equivalent.
\begin{enumerate}[leftmargin=*,itemsep=\topsep,label=\alph*.]
\item $a^2N_{k+1}=aN_k=N_{k-1}$.
\item There exists $v\in V$ so that $h_K(u)=h_{aL+v}(u)$ for all
$k$-extreme $u\in V$.
\end{enumerate}
\end{prop}

Before we prove Proposition \ref{prop:geom}, let us also formulate a 
geometric characterization of the vanishing condition.

\begin{lem}[Vanishing: geometric characterization]
\label{lem:vanishgeom}
The following are equivalent.
\begin{enumerate}[leftmargin=*,itemsep=\topsep,label=\alph*.]
\item $N_k=0$.
\item $\dim K<n-k$ or $\dim L<k-1$ or $\dim(K+L)<n-1$.
\end{enumerate}
\end{lem}

\begin{proof}
This is immediate from Lemma \ref{lem:mvol} and
\cite[Theorem 5.1.8]{Sch14}.
\end{proof}

We now complete the proof of Proposition \ref{prop:geom}.

\begin{proof}[Proof of Proposition \ref{prop:geom}]
We begin by observing that a $k$-extreme vector $u$ is called
$(B,K[n-k-1],L[k-2])$-extreme in the terminology of
\cite[Lemma 2.3]{SvH23}. Thus condition (b) implies, using
\cite[Lemma 2.8]{SvH23} and translation invariance, that
$$
	S_{K,K[n-k-1],L[k-2]} = a\,S_{L,K[n-k-1],L[k-2]},
$$
where 
$S_{C_1,\ldots,C_{n-2}}$ denotes the $(n-1)$-dimensional mixed area 
measure. Integrating this identity against $h_K$ and $h_L$ yields
condition (a) by \cite[eq.\ (2.1)]{SvH23} and Lemma~\ref{lem:mvol}. Thus 
we have proved the implication (b)$\Rightarrow$(a).

Now suppose condition (a) holds. Then $N_k>0$ implies $N_{k-1}>0$ and 
$N_{k+1}>0$ as well. Lemma \ref{lem:vanishgeom} yields $\dim K \ge 
n-k+1$, $\dim L\ge k$, and $\dim(K+L)=n-1$, so that the bodies 
$(K[n-k-1],L[k-2])$ are supercritical in the terminology of 
\cite[Definition 2.14]{SvH23}. Therefore, by \cite[Corollary~2.16]{SvH23} 
and Lemma~\ref{lem:mvol}, condition (a) implies that there exist $a'>0$ 
and $v\in V$ so that $h_K(u)=h_{a'L+v}(u)$ for all $k$-extreme vectors 
$u\in V$. It remains to note that we must have $a'=a$ by the implication 
(b)$\Rightarrow$(a). Thus we have proved the 
implication (a)$\Rightarrow$(b). 
\end{proof}

Proposition \ref{prop:geom} provides a complete characterization of the 
equality cases of the Kahn-Saks inequality in terms of the geometry of the 
polytopes $K$ and $L$. It is far from clear, however, how the 
combinatorial structure of $P$ emerges from this geometric information. 
Understanding the latter is the main challenge in the proof of our
main results, which will be addressed in the following sections.

\section{Facial structure}
\label{sec:face}

To exploit Proposition \ref{prop:geom}, we must first develop the 
connection between $k$-extreme vectors and the combinatorial structure of 
the poset. Understanding which vectors are $k$-extreme requires us to 
compute the dimensions of faces of the polytopes $K$, $L$, and $K+L$. In 
this section we will introduce a basic recipe that enables us to compute 
face dimensions in terms of the poset structure. With this recipe in hand, 
we will systematically catalogue the $k$-extreme directions that will be 
relevant in the subsequent analysis. It will turn out that some of the 
combinatorial conditions for vectors to be $k$-extreme are connected to 
the properties in Definition~\ref{defn:conditions}, which explains how 
these arise in the analysis. Another simple application of the basic 
recipe will yield a proof of Lemma \ref{lem:vanish} using Lemma 
\ref{lem:vanishgeom}.

Throughout this section, we denote by $\affp K$ the centered affine hull 
of a set $K$, that is, $\affp K := \sspan(K-K)$.
In particular, note that $\affp(K+L)=\affp K + \affp L$ and that
$\dim K = \dim(\affp K)$ for any convex bodies $K,L$.

\subsection{The basic recipe}

By definition, the polytopes $K,L$ and their faces are defined as 
intersections of the order polytope $O_P$ with certain affine hyperplanes. 
It is straightforward to deduce an upper bound on their dimensions 
from this information and the poset structure. For example, in 
$K=O_P\cap\{t\in\mathbb{R}^P:t_y=t_x\}$, the constraint $t_y=t_x$ forces 
$t_x=t_z=t_y$ for $x<z<y$ by the definition of $O_P$. Thus
$$
	K \subset H=\{t\in\mathbb{R}^P: t_x=t_z=t_y\mbox{ for all }
	z\in P_{x<\cdot<y}\},
$$
which yields $\dim K \le \dim H= n-1-|P_{x<\cdot<y}|$.

To prove a matching lower bound, however, we must show that there are no 
other constraints than those accounted for in the upper bound. To 
make this reasoning formal, we introduce the following definition.

\begin{defn}
For any linear extension $f:P\to[n]$, define the simplex
$$
	\Delta_f := \{t\in[0,1]^P:t_{f^{-1}(1)} \le
	t_{f^{-1}(2)}\le \cdots \le t_{f^{-1}(n)}\}.
$$
Thus $O_P=\bigcup_f \Delta_f$, where
the union is over all linear extensions of $P$.
\end{defn}

To illustrate how this will be used, suppose there exists a linear 
extension $f$ of $P$ so that $P_{x<\cdot<y}$ are the \emph{only} elements 
of $P$ that lie between $x,y$ in the linear order defined by $f$ (the 
existence of such a linear extension will be shown in Lemma 
\ref{lem:linxmid}). If we set $i=f(x)$ and $j=f(y)$, then
$P_{x<\cdot<y}=\{f^{-1}(k):i<k<j\}$ and
\begin{multline*}
	\Delta_f\cap\{t\in\mathbb{R}^P:t_y=t_x\} =\\
	\{t\in[0,1]^P:t_{f^{-1}(1)} \le\cdots
	\le t_{f^{-1}(i)}=\cdots=t_{f^{-1}(j)}\le\cdots
	\le t_{f^{-1}(n)}\}\subseteq K\subset H.
\end{multline*}
As the above set has nonempty relative interior in $H$ (obtained
by replacing the inequalities by strict inequalities), it
follows immediately that $\affp K = H$, and we can therefore compute
$\dim K = n-1-|P_{x<\cdot<y}|$.

The above example shows that the key to computing the dimension is to 
establish the existence of linear extensions that minimally satisfy a 
given set of constraints. We presently prove a number of lemmas that will 
be used to construct such minimal linear extensions. The basic recipe 
illustrated above will be applied systematically in the remainder of this 
section to compute face dimensions.

In the following, we take for granted the standard fact that every poset 
has at least one linear extension (such an extension may be found, for 
example, by iteratively choosing the next element of the linear order 
to be minimal among the poset elements that have not yet been ordered).
We construct three kinds of minimal linear extensions that will be used 
repeatedly below.

\begin{lem}
\label{lem:linxmid}
There exists a linear extension $f:P\to[n]$ so that
$$
	\{z\in P:f(x)<f(z)<f(y)\}=P_{x<\cdot<y}.
$$
\end{lem}

\begin{proof}
Choose a linear extension $f:P\to[n]$ that minimizes $f(y)-f(x)$ among all 
linear extensions of $P$. By definition, we have $f(x)<f(z)<f(y)$ for
$z\in P_{x<\cdot<y}$. On the other hand, if $f(x)<f(z)<f(y)$ for some 
$z\not\in P_{x<\cdot<y}$, then it must be the case that either 
$z\parallel x$ or $z\parallel y$. We aim to show this cannot occur.

To this end, consider the element $z\parallel x$ so that
$f(x)<f(z)<f(y)$ and $f(z)$ is as small as possible. Then for every
$f(x)<f(z')<f(z)$, we must have $x<z'$ and thus $z'\parallel z$.
We can therefore obtain another linear extension $g$ of $P$ by moving the 
element $z$ to have rank right below $x$, while keeping the order of the 
remaining elements as in $f$. As $g(y)-g(x)=f(y)-f(x)-1$, this contradicts 
minimality of $f(y)-f(x)$.
Thus we have ruled out the existence of $f(x)<f(z)<f(y)$ with 
$z\parallel x$. A completely analogous argument rules out $z\parallel y$, 
concluding the proof.
\end{proof}

\begin{lem}
\label{lem:linxend}
The following hold.
\begin{enumerate}[leftmargin=*,itemsep=\topsep,label=\alph*.]
\item Let $S,T\subset P$ satisfy $S\cap T=\varnothing$, $S$ is a lower set 
\emph{(}$P_{<z}\subseteq S$ for all $z\in S$\emph{)}, and $T$ is an 
upper set \emph{(}$P_{>z}\subseteq T$ for all $z\in T$\emph{)}.
Then there is a linear extension $f:P\to[n]$ so that
$\{z\in P:f(z)\le|S|\}=S$ and 
$\{z\in P:f(z)> n-|T|\}=T$.
\item There exists a linear extension $f:P\to[n]$ so that
$$
	\{z\in P:f(z)<f(x)\}=P_{<x},\qquad
	\{z\in P:f(z)>f(y)\}=P_{>y}.
$$
\end{enumerate}
\end{lem}

\begin{proof}
As $S$ is a lower set and $T$ is an upper set, any $z\in S$, $z'\in 
P\backslash(S\cup T)$, and $z''\in T$ must satisfy
$z'\not\le z$ and $z'\not\ge z''$. Therefore, if we define a linear 
ordering $f$ 
of $P$ by choosing its $|S|$ smallest 
elements to be any linear extension of $S$, its $|T|$ largest elements to 
be any linear extension of $T$, and the remaining elements to be any 
linear extension of $P\backslash (S\cup T)$, then $f$ is itself a linear
extension of $P$. This proves part (a). Part (b) follows by
choosing $S=P_{\le x}$ and $T=P_{\ge y}$.
\end{proof}

\begin{lem}
\label{lem:linxcover}
Let $z_1,z_2\in P$ such that $z_1\lessdot z_2$, and let $f:P\to[n]$ be any 
linear extension. Then there is a linear extension $f':P\to[n]$ so 
that $f'(z_2)-f'(z_1)=1$, and so that $f'(z)=f(z)$ whenever $f(z)<f(z_1)$ 
or $f(z)>f(z_2)$.
\end{lem}

\begin{proof}
Choose a linear extension $f':P\to[n]$ that minimizes $f'(z_2)-f'(z_1)$ 
among all linear extensions of $P$ such that $f'(z)=f(z)$ when
$f(z)<f(z_1)$ or $f(z)>f(z_2)$. If there exists $f'(z_1)<f'(z)<f'(z_2)$,
then it must be the case that either $z\parallel z_1$ or $z\parallel z_2$
as $z_1\lessdot z_2$.
This entails a contradiction as in the proof of Lemma \ref{lem:linxmid}.
\end{proof}

\subsection{The vanishing condition}

As a first illustration of the basic recipe introduced above, we will now 
give a short proof of Lemma \ref{lem:vanish} using Lemma 
\ref{lem:vanishgeom}. Before we do so, let us define a shorthand notation 
for the linear spaces that will appear repeatedly throughout the remainder 
of this section.

\begin{defn}
For any disjoint subsets $S_1,\ldots,S_k\subseteq P$, we define
$$
	\llangle S_1,\ldots,S_k\rrangle :=
	\sspan\big\{
	{\textstyle 
	\sum_{z\in S_1}e_z,\ldots,\sum_{z\in S_k}e_z,
	\mathbb{R}^{P\backslash\{S_1\cup\cdots\cup S_k\}}
	}
	\big\}.
$$
Note that $\dim \llangle S_1,\ldots,S_k\rrangle = 
n+k-\sum_{i=1}^k|S_i|$.
\end{defn}

We now proceed to the proof.

\begin{proof}[Proof of Lemma \ref{lem:vanish}]
As $K=O_P\cap\{t\in\mathbb{R}^P:t_y=t_x\}$, every $t\in K$ must satisfy
$t_x=t_z=t_y$ for all $z\in P_{x<\cdot<y}$. Thus 
$\affp K \subseteq \llangle P_{x\le\cdot\le y}\rrangle$.
On the other hand, if $f$ is the linear extension provided by Lemma 
\ref{lem:linxmid}, we have 
$$
	\llangle P_{x\le\cdot\le y}\rrangle = 
	\affp(\Delta_f\cap\{t\in\mathbb{R}^P:t_y=t_x\})\subseteq \affp K.
$$
We have therefore shown that 
$\affp K = \llangle P_{x\le\cdot\le y}\rrangle$.

Similarly, as
$L=O_P\cap\{t\in\mathbb{R}^P:t_x=0,t_y=1\}$, every $t\in L$ satisfies
$t_z=t_x=0$ for all $z\in P_{<x}$ and $t_z=t_y=1$ for all $z\in P_{>y}$.
Thus $\affp L\subseteq \mathbb{R}^{P\backslash (P_{\le x}\cup 
P_{\ge y})}$. On the other hand, if $f$ is the linear extension provided 
by Lemma \ref{lem:linxend}(b), we have 
$$
	\mathbb{R}^{P\backslash (P_{\le x}\cup P_{\ge y})} =
	\affp(\Delta_f\cap\{t\in\mathbb{R}^P:t_x=0,t_y=1\})
	\subseteq \affp L.
$$ 
We have therefore shown that
$\affp L = \mathbb{R}^{P\backslash (P_{\le x}\cup P_{\ge y})}$.

We can now immediately conclude that
\begin{alignat*}{3}
	& \dim K &&=
	n+1-|P_{x\le\cdot\le y}| &&=n-1-|P_{x<\cdot<y}|, \\
	&\dim L &&= n - |P_{\le x}\cup P_{\ge y}| &&=
	n-2-|P_{<x}|-|P_{>y}|.
\end{alignat*}
But as $\affp(K+L)=\affp K+\affp L =
\llangle \{x,y\}\rrangle$, we also have $\dim(K+L)=n-1$.
The conclusion of Lemma \ref{lem:vanish} now follows 
from Lemma \ref{lem:vanishgeom}.
\end{proof}

The following corollary of Lemma \ref{lem:vanish} will be used frequently.

\begin{cor}
\label{cor:eqpos}
If $N_k>0$ and $N_k^2=N_{k-1}N_{k+1}$, then
$$
	|P_{x<\cdot<y}|+1 < k < n-1-|P_{<x}|-|P_{>y}|.
$$
\end{cor}

\begin{proof}
The assumption implies that $N_{k-1}>0$ and $N_{k+1}>0$. The conclusion 
follows immediately from Lemma \ref{lem:vanish}.
\end{proof}

\subsection{Extreme vectors}

We now turn to the main task of this section, which is to understand which 
vectors are $k$-extreme. More precisely, as we will only use the 
implication (a)$\Rightarrow$(b) of Proposition \ref{prop:geom}, we do not 
need to (and will not) characterize \emph{all} $k$-extreme vectors; it 
suffices to find vectors that carry significant combinatorial information. 
To this end, we will consider the following vectors.

\begin{defn}
\label{defn:candidates}
A vector $u\in V$ is called a
\begin{enumerate}[leftmargin=*,itemsep=\topsep,label=\alph*.]
\item \emph{coordinate vector} if $u=\pm e_z$ for $z\in 
P\backslash\{x,y\}$;
\item \emph{transition vector} if $u=e_{zz'} := e_z-e_{z'}$ for $z,z'\in 
P\backslash\{x,y\}$;
\item \emph{anchor vector} if 
$u=e_{zxy} := e_z-\frac{e_x+e_y}{2}$ or
$e_{xyz} := \frac{e_x+e_y}{2}-e_z$ for $z\in P\backslash\{x,y\}$.
\end{enumerate}
\end{defn}

The motivation for considering these particular vectors is 
straightforward. By definition, a vector $u\in V$ is $k$-extreme if the 
associated faces of $K$, $L$, and $K+L$ are sufficiently high-dimensional. 
As $K$ and $L$ are slices of the order polytope $O_P$, the most natural 
candidates for such vectors are the normal directions of the facets (i.e., 
the highest-dimensional faces) of $O_P$. It is well known \cite[\S 
1]{Sta86} that $u$ is a facet normal of $O_P$ if and only if $u=\pm e_z$ 
for a maximal (minimal) element $z$ of $P$, or if $u=e_z-e_{z'}$ for 
$z\lessdot z'$. This motivates the consideration of coordinate and 
transition vectors for $z,z'\not\in \{x,y\}$. Note, however, that (for 
example) $e_z-e_x\not\in V$ and thus cannot be $k$-extreme. In these cases 
we consider instead the projections of such vectors onto $V$, which are 
the anchor vectors.

\begin{rem}
The above logic suggests we should consider one additional case 
$u=\pm\frac{e_x+e_y}{2}$, which corresponds to projecting $\pm e_x$ or 
$\pm e_y$ onto $V$. In principle such vectors may indeed be needed in 
the analysis when $x$ is a minimal element of $P$ or when $y$ is a maximal 
element of $P$. 
However, in the proof of our main results we will be able to assume 
without loss of generality this is not the case, as inserting a new 
element in $P$ that is smaller (larger) than all the other elements does 
not change the numbers $N_k$. This simple observation does not 
make any fundamental difference to the proof, but slightly shortens the 
analysis in a few places.
\end{rem}

We now proceed to systematically investigate the combinatorial conditions
for coordinate, transition, and anchor vectors to be $k$-extreme.

\subsubsection{Coordinate vectors}

We begin with a basic observation.

\begin{lem}
\label{lem:hcoord}
For $z\in P\backslash\{x,y\}$, the following hold.
\begin{enumerate}[leftmargin=*,itemsep=\topsep,label=\alph*.]
\item If $z$ is a maximal element of $P$, then $h_K(e_z)=h_L(e_z)=1$.
\item If $z$ is a minimal element of $P$, then $h_K(-e_z)=h_L(-e_z)=0$.
\end{enumerate}
\end{lem}

\begin{proof}
If $z$ is a maximal element of $P\backslash\{x,y\}$, then $e_z\in K$.
Thus
$$
	1=\langle e_z,e_z\rangle
	\le
	h_K(e_z) = \sup_{t\in K}\langle e_z,t\rangle \le 1,
$$
where we used that $K\subseteq[0,1]^P$ in the second inequality.
Therefore $h_K(e_z)=1$. Now let $t'\in\mathbb{R}^P$ be
defined by $t_{z'}'=1_{z'\in P_{\ge y}\cup \{z\}}$. As $z$ is maximal,
$t'\in L$ and thus
$$
	1 = \langle e_z,t'\rangle \le 
	h_L(e_z) = \sup_{t\in L}\langle e_z,t\rangle \le 1.
$$
We have therefore shown that $h_L(e_z)=1$, concluding the proof of part 
(a). The proof of part (b) is completely analogous.
\end{proof}

We can now exhibit $k$-extreme coordinate vectors.

\begin{lem}
\label{lem:kxcoord}
Let $N_k^2=N_{k-1}N_{k+1}>0$.
For $z\in P\backslash\{x,y\}$, the following hold.
\begin{enumerate}[leftmargin=*,itemsep=\topsep,label=\alph*.]
\item If $z$ is a maximal element of $P$, then $e_z$ is $k$-extreme.
\item If $z$ is a minimal element of $P$, then $-e_z$ is $k$-extreme.
\end{enumerate}
\end{lem}

\begin{proof}
Let $z$ be a maximal element of $P$.
Then Lemma \ref{lem:hcoord} yields
\begin{align*}
	F(K,e_z) &= O_P \cap
	\{t\in\mathbb{R}^P:t_y=t_x,t_z=1\},\\
	F(L,e_z) &= O_P \cap
	\{t\in\mathbb{R}^P:t_x=0,t_y=t_z=1\}.
\end{align*}
Let us compute the affine hulls.
\begin{enumerate}[leftmargin=*,itemsep=\topsep,label=\textbullet]
\item
Clearly $\affp F(K,e_z) \subseteq \llangle 
P_{x\le\cdot\le y}\rrangle\cap\mathbb{R}^{P\backslash\{z\}}$ (note that 
$z\not\in P_{x\le\cdot\le y}$ as $z$ is maximal). Now let $f$ be the 
linear ordering of $P$ obtained by applying Lemma \ref{lem:linxmid} to 
$P\backslash\{z\}$ and setting $f(z)=n$. As $z$ is maximal, $f$ is a 
linear extension of $P$. Thus
$$
	\llangle P_{x\le\cdot\le y}\rrangle
	\cap \mathbb{R}^{P\backslash\{z\}}
	=
	\affp(\Delta_f\cap\{t\in\mathbb{R}^P:t_y=t_x,t_z=1\})
	\subseteq \affp F(K,e_z).
$$
\item 
Clearly $\affp F(L,e_z) \subseteq
\mathbb{R}^{P\backslash(P_{\le x}\cup P_{\ge y}\cup\{z\})}$. 
Now let $f$ 
be the linear extension of $P$ obtained by applying Lemma 
\ref{lem:linxend}(b) to $P\backslash\{z\}$ and setting $f(z)=n$. Then
$$
	\mathbb{R}^{P\backslash(P_{\le x}\cup P_{\ge y}\cup\{z\})} 
	=
	\affp(\Delta_f\cap\{t\in\mathbb{R}^P:t_x=0,t_y=t_z=1\})
	\subseteq \affp F(L,e_z).
$$
\end{enumerate}
Thus $\affp F(K,e_z) = \llangle P_{x\le\cdot\le 
y}\rrangle\cap \mathbb{R}^{P\backslash\{z\}}$ and
$\affp F(L,e_z) = \mathbb{R}^{P\backslash(P_{\le 
x}\cup P_{\ge y}\cup\{z\})}$,
which implies $\affp F(K+L,e_z) =
\llangle \{x,y\}\rrangle\cap \mathbb{R}^{P\backslash\{z\}}$. Therefore
\begin{align*}
	\dim F(K,e_z) &= n-2-|P_{x<\cdot<y}|,\\
	\dim F(L,e_z) &= n-2-|P_{<x}|-|P_{>y}\cup\{z\}|,\\
	\dim F(K+L,e_z) &= n-2.
\end{align*}
It follows readily from
Corollary \ref{cor:eqpos} that $e_z$ is $k$-extreme, which concludes the 
proof of part (a). The proof of part (b) is completely analogous.
\end{proof}

\subsubsection{Transition vectors}

We now turn our attention to the transition vectors. As above, we 
must first identify the corresponding supporting hyperplanes.

\begin{lem}
\label{lem:htrans}
Consider $z,z'\in P\backslash\{x,y\}$ such that $z<z'$, and assume that 
either $z\not\in P_{<x}$ or $z'\not\in P_{>y}$.
Then we have $h_K(e_{zz'}) = h_L(e_{zz'})=0$.
\end{lem}

\begin{proof}
As $z<z'$, any $t\in O_P$ must satisfy $\langle e_{zz'},t\rangle = 
t_z-t_{z'} \le 0$. Thus
$$
	0\le 
	h_K(e_{zz'}) = \sup_{t\in K}\langle e_{zz'},t\rangle \le 0,
$$
where we used that $0\in K$. On the other hand, define $t\in L$ by 
$t_{z''}=1_{z''\not\le x}$ if $z\not\in P_{<x}$, and by
$t_{z''}=1_{z''\ge y}$ otherwise. Then
$$
	0 = \langle e_{zz'},t\rangle \le
	h_L(e_{zz'}) = \sup_{t'\in L}\langle e_{zz'},t'\rangle \le 0,
$$
concluding the proof.
\end{proof}

We can now exhibit $k$-extreme transition vectors.

\begin{lem}
\label{lem:kxtrans}
Assume $N_k^2=N_{k-1}N_{k+1}>0$.
Let $z,z'\in P\backslash\{x,y\}$ such that $z\lessdot z'$. 
Then the transition vector $e_{zz'}$ is $k$-extreme in each of the 
following situations:
\begin{enumerate}[leftmargin=*,itemsep=\topsep,label=\alph*.]
\item $z,z'\in P_{<x}$, or $z,z'\in P_{>y}$.
\item $z,z'\in P_{>x,\parallel y}$, or $z,z'\in P_{<y,\parallel x}$.
\item $z,z'\in P_{x<\cdot<y}$.
\item $z\in P_{>x,\parallel y}$ and $P_{>z}\subseteq P_{>y}$, or
$z'\in P_{<y,\parallel x}$ and $P_{<z'}\subseteq P_{<x}$.
\item $z\in P_{\parallel x,\parallel y}$ 
and $P_{>z}\subseteq P_{>y}$, or
$z'\in P_{\parallel x,\parallel y}$ and 
$P_{<z'}\subseteq P_{<x}$.
\item $z\in P_{\parallel x,\parallel y}$ and $z'\not\in P_{>y}$, 
or $z'\in P_{\parallel x,\parallel y}$ and $z\not\in P_{<x}$.
\item $z\in P_{x<\cdot<y}$, $z'\in P_{>x,\parallel y}$, and
$|P_{x<\cdot<z'}\cup P_{x<\cdot<y}|\le k-1$.
\item $z\in P_{<y,\parallel x}$, $z'\in P_{x<\cdot<y}$, and
$|P_{z<\cdot<y}\cup P_{x<\cdot<y}|\le k-1$.
\item $z\in P_{<y,\parallel x}$, $z'\in P_{>x,\parallel y}$, and
$|P_{z<\cdot<y}\cup P_{x<\cdot<z'}\cup P_{x<\cdot<y}|\le k-2$.
\end{enumerate}
\end{lem}

\begin{proof}
We begin by noting that in all cases, we have
\begin{align*}
	F(K,e_{zz'}) &= O_P\cap\{ t\in\mathbb{R}^P:
	t_y=t_x,t_z=t_{z'}\}, \\
	F(L,e_{zz'}) &= O_P\cap\{ t\in\mathbb{R}^P:
	t_x=0,t_y=1,t_z=t_{z'}\}
\end{align*}
by Lemma \ref{lem:htrans}. We now consider each case separately.

\medskip

(a) Assume that $z,z'\in P_{<x}$ (the proof for $z,z'\in P_{>y}$ 
is completely analogous). Clearly $F(L,e_{zz'})=L$ and
$\affp F(K,e_{zz'})\subseteq \llangle P_{x\le \cdot\le y},
\{z,z'\}\rrangle$. On the other hand, by Lemmas \ref{lem:linxmid} and
\ref{lem:linxcover}, there is a linear extension $f$ of $P$ so that
the only elements between $x$ and $y$ are $P_{x<\cdot<y}$ and such that 
$z,z'$ are adjacent in the linear order defined by $f$. Applying the basic 
recipe yields $\affp F(K,e_{zz'})=\llangle P_{x\le \cdot\le y},
\{z,z'\}\rrangle$. As $\affp L = \mathbb{R}^{P\backslash(P_{\le x}\cup 
P_{\ge y})}$ by the proof of Lemma \ref{lem:vanish}, we obtain
\begin{align*}
	\dim F(K,e_{zz'}) &= n-2-|P_{x<\cdot<y}|,\\
	\dim F(L,e_{zz'}) &= n-2-|P_{<x}|-|P_{>y}|,\\
	\dim F(K+L,e_{zz'}) &= n-2,
\end{align*}
where we use $\affp F(K+L,e_{zz'}) = 
\affp F(K,e_{zz'}) + \affp F(L,e_{zz'}) =
\llangle \{z,z'\},\{x,y\}\rrangle$.
It follows readily from
Corollary \ref{cor:eqpos} that $e_{zz'}$ is $k$-extreme.

\medskip

(b) Assume $z,z'\in P_{>x,\parallel y}$ (the proof for $z,z'\in 
P_{<y,\parallel x}$ is completely analogous).
Then $\affp F(K,e_{zz'})=\llangle P_{x\le \cdot\le 
y},\{z,z'\}\rrangle$ as in part (a). Moreover, we clearly have
$\affp F(L,e_{zz'}) \subseteq \llangle \{z,z'\}\rrangle\cap 
\mathbb{R}^{P\backslash(P_{\le x}\cup P_{\ge y})}$. By 
Lemmas \ref{lem:linxend}(b) and \ref{lem:linxcover}, there is a linear 
extension $f$ of $P$ in which
the only elements less than $x$ are $P_{<x}$, the only elements greater 
than $y$ are $P_{>y}$, and $z,z'$ are adjacent. Applying the basic
recipe yields $\affp F(L,e_{zz'}) = \llangle \{z,z'\}\rrangle\cap
\mathbb{R}^{P\backslash(P_{\le x}\cup P_{\ge y})}$. We therefore obtain
\begin{align*}
	\dim F(K,e_{zz'}) &= n-2-|P_{x<\cdot<y}|,\\
	\dim F(L,e_{zz'}) &= n-3-|P_{<x}|-|P_{>y}|, \\
	\dim F(K+L,e_{zz'}) &= n-2,
\end{align*}
where we use $\affp F(K+L,e_{zz'}) =
\affp F(K,e_{zz'}) + \affp F(L,e_{zz'}) =
\llangle \{z,z'\},\{x,y\}\rrangle$.
It follows readily from
Corollary \ref{cor:eqpos} that $e_{zz'}$ is $k$-extreme.

\medskip

(c) Clearly $F(K,e_{zz'})=K$, while
$\affp F(L,e_{zz'}) = \llangle \{z,z'\}\rrangle\cap
\mathbb{R}^{P\backslash(P_{\le x}\cup P_{\ge y})}$ as in part (b).
As $\affp K = \llangle P_{x\le\cdot\le y}\rrangle$ by the proof of Lemma 
\ref{lem:vanish}, we obtain
\begin{align*}
	\dim F(K,e_{zz'}) &= n-1-|P_{x<\cdot<y}|,\\
	\dim F(L,e_{zz'}) &= n-3-|P_{<x}|-|P_{>y}|, \\
	\dim F(K+L,e_{zz'}) &= n-2,
\end{align*}
where we use $\affp F(K+L,e_{zz'}) = \affp F(K,e_{zz'}) + \affp 
F(L,e_{zz'}) = \llangle \{z,z'\},\{x,y\}\rrangle$.
It follows readily from
Corollary \ref{cor:eqpos} that $e_{zz'}$ is $k$-extreme.

\medskip

(d) Assume $z\in P_{>x,\parallel y}$ and $P_{>z}\subseteq P_{>y}$ (the 
other case is
completely analogous). Then $\affp F(K,e_{zz'})=\llangle P_{x\le \cdot\le
y},\{z,z'\}\rrangle$ as in part (a). Moreover, we clearly have
$\affp F(L,e_{zz'}) \subseteq \mathbb{R}^{P\backslash(P_{\le x}\cup
P_{\ge y}\cup\{z\})}$. Applying Lemma \ref{lem:linxend}(a) to
$S=P_{\le x}$ and $T=P_{\ge y}\cup\{z\}$ yields
$\affp F(L,e_{zz'}) = \mathbb{R}^{P\backslash(P_{\le x}\cup
P_{\ge y}\cup\{z\})}$
using the basic recipe. Thus
\begin{align*}
	\dim F(K,e_{zz'}) &= n-2-|P_{x<\cdot<y}|,\\
	\dim F(L,e_{zz'}) &= n-3-|P_{<x}|-|P_{>y}|, \\
	\dim F(K+L,e_{zz'}) &= n-2,
\end{align*}
where we use $\affp F(K+L,e_{zz'}) =
\affp F(K,e_{zz'}) + \affp F(L,e_{zz'}) =
\llangle \{z,z'\},\{x,y\}\rrangle$.
It follows readily from
Corollary \ref{cor:eqpos} that $e_{zz'}$ is $k$-extreme.

\medskip

(e) Assume $z\in P_{\parallel x,\parallel y}$ and $P_{>z}\subseteq P_{>y}$ 
(the other case is completely analogous). Then $\affp F(L,e_{zz'}) = 
\mathbb{R}^{P\backslash(P_{\le x}\cup P_{\ge y}\cup\{z\})}$ as in part 
(d). Moreover, we clearly have 
$\affp F(K,e_{zz'})\subseteq \llangle P_{x\le \cdot\le
y},\{z,z'\}\rrangle$. To prove equality we proceed as in part (a), but we 
must be more careful in constructing the linear 
extension $f$.

By Lemma \ref{lem:linxmid}, there is a linear extension $f$ of $P$ in 
which the only elements between $x$ and $y$ are $P_{x<\cdot<y}$. If 
$f(z)<f(x)$, choose $z_1\ge z$, $f(z)\le f(z_1)<f(x)$ with maximal 
$f(z_1)$. We claim that any $f(z_1)<f(z_2)\le f(y)$ must satisfy 
$z_2\parallel z_1$: otherwise $z_2>z_1\ge z$, which contradicts maximality 
of $f(z_1)$ if $f(z_2)<f(x)$ and contradicts $z\in P_{\parallel 
x,\parallel y}$ if $f(x)\le f(z_2)\le f(y)$ (as the latter implies $z\in 
P_{x\le \cdot\le y}$). We can therefore obtain a new linear extension of 
$P$ by moving $z_1$ right above $y$ in the linear order defined by $f$, 
while keeping the remaining ordering fixed. By iterating this process, we 
can always modify the original linear extension $f$ so that $f(z)>f(y)$. 
Applying Lemma \ref{lem:linxcover} to the latter yields a linear extension 
in which the only elements between $x$ and $y$ are $P_{x<\cdot<y}$ and 
$z,z'$ are adjacent. We can now apply the basic recipe to conclude that 
$\affp F(K,e_{zz'})=\llangle P_{x\le \cdot\le y},\{z,z'\}\rrangle$.

The rest of the proof of part (e) is identical to that of part (d).

\medskip

(f) Assume $z\in P_{\parallel x,\parallel y}$ and $z'\not\in P_{>y}$ (the 
other case is completely analogous). Then $\affp F(K,e_{zz'})=\llangle 
P_{x\le \cdot\le y},\{z,z'\}\rrangle$ as in part (e). Moreover, as $z\in 
P_{\parallel x,\parallel y}$ also implies $z'\not\in P_{<x}$, we obtain 
$\affp F(L,e_{zz'})= \llangle \{z,z'\}\rrangle\cap 
\mathbb{R}^{P\backslash(P_{\le x}\cup P_{\ge y})}$ as in part (b).
The rest of the proof of part (f) is identical to that of part (b).

\medskip

(g) We have  $\affp F(L,e_{zz'}) = \llangle \{z,z'\}\rrangle\cap
\mathbb{R}^{P\backslash(P_{\le x}\cup P_{\ge y})}$ as in part (b).
Moreover, we clearly have $\affp F(K,e_{zz'}) \subseteq
\llangle P_{x\le\cdot\le y}\cup P_{x\le\cdot\le z'}\rrangle$.

By Lemma \ref{lem:linxmid}, there is a linear extension $f$ of $P$ in 
which the only elements between $x$ and $y$ are $P_{x<\cdot<y}$. As $z'\in 
P_{>x,\parallel y}$ we must have $f(y)<f(z')$. 

Choose $z_1\not\in P_{<z'}$ with $f(y)<f(z_1)<f(z')$, if it exists, so 
that $f(z_1)$ is maximal. Then any $f(z_1)<f(z_2)\le f(z')$ must satisfy 
$z_2\parallel z_1$: otherwise $z_1<z_2\in P_{\le z'}$ yields a 
contradiction. We can therefore obtain a new linear extension of $P$ by 
moving $z_1$ right above $z'$ while keeping the remaining ordering fixed. 
Iterating this process, we can ensure that $z''\in P_{<z'}$ for all 
$f(y)<f(z'')<f(z')$.

Now choose $z_2\not\in P_{>x}$ with $f(y)<f(z_2)<f(z')$, if it exists, so 
that $f(z_2)$ is minimal. Then any $f(x)\le f(z_1)<f(z_2)$ must satisfy 
$z_1\parallel z_2$: otherwise $x\le z_1<z_2$ or $y=z_1<z_2$, which 
contradict $z_2\not\in P_{>x}$ and $z'\in P_{>x,\parallel y}$ (as we
already ensured above that $z_2\in P_{<z'}$). 
Thus we can move $z_2$ right below $x$ while keeping the 
remaining ordering fixed. Iterating this process, we can always modify the 
original linear extension $f$ so that $z''\in P_{x<\cdot<z'}$ for all 
$f(y)<f(z'')<f(z')$. 
In particular, the resulting linear extension 
satisfies
$$
	\{z''\in P:f(x)\le f(z'')\le f(z')\}=P_{x\le\cdot\le y}\cup 
	P_{x\le\cdot\le z'}.
$$
Applying the basic recipe yields $\affp F(K,e_{zz'}) =
\llangle P_{x\le\cdot\le y}\cup P_{x\le\cdot\le z'}\rrangle$. Thus
\begin{align*}
	\dim F(K,e_{zz'}) &= n-2-|P_{x<\cdot<y}\cup P_{x<\cdot<z'}|,\\
	\dim F(L,e_{zz'}) &= n-3-|P_{<x}|-|P_{>y}|, \\
	\dim F(K+L,e_{zz'}) &= n-2,
\end{align*}
where we use $\affp F(K+L,e_{zz'}) =
\affp F(K,e_{zz'}) + \affp F(L,e_{zz'}) =
\llangle \{z,z'\},\{x,y\}\rrangle$.
It follows from
Corollary \ref{cor:eqpos} and the assumption
that $e_{zz'}$ is $k$-extreme.

\medskip

(h) The proof is completely analogous to that of part (g).

\medskip

(i) We have $\affp F(L,e_{zz'}) = \llangle \{z,z'\}\rrangle\cap 
\mathbb{R}^{P\backslash(P_{\le x}\cup P_{\ge y})}$ as in part (b). Next, 
note that any $t\in F(K,e_{zz'})$ must satisfy $t_z\le t_y=t_x\le 
t_{z'}=t_z$ as $z'>x$ and $z<y$. We therefore clearly have $\affp 
F(K,e_{zz'}) \subseteq \llangle P_{x\le\cdot\le y}\cup P_{x\le\cdot\le 
z'}\cup P_{z\le\cdot\le y}\rrangle$.

To prove equality we reason similarly as in part (g). By Lemma 
\ref{lem:linxmid}, there is a linear extension $f$ of $P$ in which the 
only elements between $x$ and $y$ are $P_{x<\cdot<y}$. As $z'>x$ and 
$z<y$, we must have $f(z)<f(x)<f(y)<f(z')$. By modifying the linear 
extension as in part (g), we can ensure that all $f(y)<f(z'')<f(z')$ 
satisfy $z''\in P_{<z'}$ and all $f(z)<f(z'')<f(x)$ satisfy $z''\in 
P_{>z}$.

Now choose $z_2\not\in P_{>x}$ with $f(y)<f(z_2)<f(z')$, if it exists, so 
that $f(z_2)$ is minimal. Then any $f(z)\le f(z_1)<f(z_2)$ must satisfy 
$z_1\parallel z_2$: otherwise $x\le z_1<z_2$, $y=z_1<z_2$, or $z\le 
z_1<z_2$, which contradict $z_2\not\in P_{>x}$, $z'\in P_{>x,\parallel 
y}$, and $z\lessdot z'$ (as we already ensured that $z_2\in P_{<z'}$).
Thus we can move $z_2$ right below $z$ while keeping the remaining 
ordering fixed. Iterating this process, we can modify the linear 
extension $f$ so that $z''\in P_{x<\cdot<z'}$ for all 
$f(y)<f(z'')<f(z')$.
A completely analogous argument ensures also that
$z''\in P_{z<\cdot<y}$ for all $f(z)<f(z'')<f(x)$.
In particular, the resulting linear extension 
satisfies
$$
	\{z''\in P:f(z)\le f(z'')\le f(z')\}=P_{x\le\cdot\le y}\cup 
	P_{x\le\cdot\le z'}\cup P_{z\le\cdot\le y},
$$
so the basic recipe yields $\affp F(K,e_{zz'}) =
\llangle P_{x\le\cdot\le y}\cup P_{x\le\cdot\le
z'}\cup P_{z\le\cdot\le y}\rrangle$. Thus
\begin{align*}
	\dim F(K,e_{zz'}) &= n-3-|P_{x<\cdot<y}\cup P_{x<\cdot<z'}
	\cup P_{z<\cdot<y}|,\\
	\dim F(L,e_{zz'}) &= n-3-|P_{<x}|-|P_{>y}|, \\
	\dim F(K+L,e_{zz'}) &= n-2,
\end{align*}
where we use $\affp F(K+L,e_{zz'}) =
\affp F(K,e_{zz'}) + \affp F(L,e_{zz'}) =
\llangle \{z,z'\},\{x,y\}\rrangle$.
It follows from
Corollary \ref{cor:eqpos} and the assumption
that $e_{zz'}$ is $k$-extreme.
\end{proof}

\subsubsection{Anchor vectors}

We conclude with an analysis of the anchor vectors.
As above, we first identify the corresponding supporting hyperplanes.

\begin{lem}
\label{lem:hanchor}
The following hold.
\begin{enumerate}[leftmargin=*,itemsep=\topsep,label=\alph*.]
\item If $z\in P_{x<\cdot<y}\cup P_{<y,\parallel x}$,
then $h_K(e_{zxy})=0$ and $h_L(e_{zxy})=\frac{1}{2}$.
\item If $z\in P_{x<\cdot<y}\cup P_{>x,\parallel y}$ then $h_K(e_{xyz})=0$ 
and $h_L(e_{xyz})=\frac{1}{2}$.
\item If $z\in P_{<x}$, then
$h_K(e_{zxy})=0$ and $h_L(e_{zxy})=-\frac{1}{2}$.
\item If $z\in P_{>y}$, then
$h_K(e_{xyz})=0$ and $h_L(e_{xyz})=-\frac{1}{2}$.
\end{enumerate}
\end{lem}

\begin{proof}
In all cases, the proof that $h_K(e_{zxy})=0$ or $h_K(e_{xyz})=0$ 
is the same as in Lemma \ref{lem:htrans}. On the other hand, note that
any $t\in L$ satisfies $\langle \frac{e_x+e_y}{2},t\rangle=\frac{1}{2}$,
so
$$
	h_L(e_{zxy}) = h_L(e_z)-\tfrac{1}{2},\qquad
	h_L(e_{xyz}) = h_L(-e_z)+\tfrac{1}{2}.
$$
For (a), define $t\in\mathbb{R}^P$ by $t_z=1_{z\in P_{\ge y}\cup
P_{\ge z}}$. As $z\not\le x$, we have $t\in L\subset [0,1]^P$, so
$$
	1 = \langle e_z,t\rangle \le h_L(e_z) =
	\sup_{t'\in L}\langle e_z,t'\rangle \le 1.
$$
Consequently, we have $h_L(e_{zxy})=\frac{1}{2}$. The proof of part (b) is 
completely analogous. For part (c), it suffices to note that
$t_z=0$ for all $t\in L$, so that $h_L(e_z)=0$ and thus
$h_L(e_{zxy})=-\frac{1}{2}$. The proof of part (d) is completely 
analogous.
\end{proof}

We can now exhibit $k$-extreme anchor vectors.

\begin{lem}
\label{lem:kxanchor}
Assume $N_k^2=N_{k-1}N_{k+1}>0$.
\begin{enumerate}[leftmargin=*,itemsep=\topsep,label=\alph*.]
\item 
If $z\in P_{x<\cdot<y}\cup P_{<y,\parallel x}$ and $z\lessdot y$, then
$e_{zxy}$ is $k$-extreme if $|P_{>z}|+|P_{<x}|\le n-k$.
\item 
If $z\in P_{x<\cdot<y}\cup P_{>x,\parallel y}$ and $x\lessdot z$, then
$e_{xyz}$ is $k$-extreme if $|P_{<z}|+|P_{>y}|\le n-k$.
\item
If $z\lessdot x$, then
$e_{zxy}$ is $k$-extreme if $|P_{z<\cdot<y}\cup\{x\}|\le k$.
\item
If $y\lessdot z$, then
$e_{xyz}$ is $k$-extreme if $|P_{x<\cdot<z}\cup\{y\}|\le k$.
\end{enumerate}
\end{lem}

\begin{proof}
We first prove (a) (the proof of part (b) is completely analogous).
We have
\begin{align*}
	F(K,e_{zxy}) &= O_P \cap \{t\in\mathbb{R}^P:
	t_y=t_z=t_x\},\\
	F(L,e_{zxy}) &= O_P \cap \{t_x=0,t_y=t_z=1\}
\end{align*}
by Lemma \ref{lem:hanchor}. 
As $z<y$, clearly $\affp F(L,e_{zxy})
\subseteq \mathbb{R}^{P\backslash(P_{\le x}\cup P_{\ge z})}$.
Applying the basic recipe with the linear extension of
Lemma \ref{lem:linxend}(a) with $S=P_{\le x}$ and
$T=P_{\ge z}$ yields
$\affp F(L,e_{zxy}) =\mathbb{R}^{P\backslash(P_{\le x}\cup P_{\ge z})}$.
Next, we clearly have $F(K,e_{zxy})=K$ if $z\in P_{x<\cdot<y}$.
If $z\in P_{<y,\parallel x}$, we obtain
$\affp F(K,e_{zxy})=\llangle P_{x\le\cdot\le y}\cup \{z\}\rrangle$
as in the proof of Lemma \ref{lem:kxtrans}(h) (note that 
$P_{z<\cdot<y}=\varnothing$ as $z\lessdot y$). In either case
\begin{align*}
	\dim F(K,e_{zxy}) &\ge n-2-|P_{x<\cdot<y}|,\\
	\dim F(L,e_{zxy}) &= n-2-|P_{<x}|-|P_{>z}|, \\
	\dim F(K+L,e_{zxy}) &= n-2,
\end{align*}
where we use $\affp F(K+L,e_{zxy}) =
\affp F(K,e_{zxy}) + \affp F(L,e_{zxy}) =
\llangle \{x,y,z\}\rrangle$.
It follows from
Corollary \ref{cor:eqpos} and the assumption
that $e_{zxy}$ is $k$-extreme.

\medskip

We now prove (d) (the proof of part (c) is completely analogous).
We have
\begin{align*}
	F(K,e_{xyz}) &= O_P \cap \{t\in\mathbb{R}^P:
	t_y=t_z=t_x\},\\
	F(L,e_{xyz}) &= O_P \cap \{t_x=0,t_y=t_z=1\} = L
\end{align*}
by Lemma \ref{lem:hanchor}. Clearly 
$\affp F(K,e_{xyz})\subseteq \llangle
P_{x\le\cdot\le y}\cup P_{x\le \cdot\le z}\rrangle$.

To prove equality we reason similarly as for Lemma \ref{lem:kxtrans}(g). 
By Lemma \ref{lem:linxmid}, there is a linear extension $f$ of $P$ in 
which the only elements between $x$ and $y$ are $P_{x<\cdot<y}$. As $y<z$ 
we must have $f(y)<f(z)$. By modifying $f$ as in the proof of Lemma 
\ref{lem:kxtrans}(g), we can ensure that all $f(y)<f(z')<f(z)$ satisfy 
$z'\in P_{<z}$.

Now choose $z_2\not\in P_{>x}$ with $f(y)<f(z_2)<f(z)$, if it exists, so 
that $f(z_2)$ is minimal. Then any $f(x)\le f(z_1)<f(z_2)$ must satisfy 
$z_1\parallel z_2$: otherwise $x\le z_1<z_2$ or $y=z_1<z_2$, which 
contradict $z_2\not\in P_{>x}$ and $y\lessdot z$ (as we already ensured 
above that $z_2\in P_{<z}$). We can now repeat the rest of the argument in 
the proof of Lemma \ref{lem:kxtrans}(g) verbatim to conclude that $\affp 
F(K,e_{xyz})=\llangle P_{x\le\cdot\le y}\cup P_{x\le \cdot\le z}\rrangle$.
Noting that
$P_{x\le\cdot\le y}\cup P_{x\le \cdot\le z}=\{x,y,z\}\cup P_{x<\cdot<z}$,
we obtain
\begin{align*}
	\dim F(K,e_{xyz}) &= n-1-|P_{x<\cdot<z}\cup\{y\}|,\\
	\dim F(L,e_{xyz}) &= n-2-|P_{<x}|-|P_{>y}|, \\
	\dim F(K+L,e_{xyz}) &= n-2,
\end{align*}
where we use $\affp F(K+L,e_{xyz}) =
\affp F(K,e_{xyz}) + \affp F(L,e_{xyz}) =
\llangle \{x,y,z\}\rrangle$.
It follows from
Corollary \ref{cor:eqpos} and the assumption
that $e_{xyz}$ is $k$-extreme.
\end{proof}

\section{Proof of the main results}
\label{sec:proof}

Now that we have described $k$-extreme vectors in combinatorial terms, 
we aim to combine this information with the equality characterization 
provided by Proposition~\ref{prop:geom} to prove the main results of this 
paper. More precisely, our analysis will be based on the following 
consequences of Proposition \ref{prop:geom}. Here and in the following, we 
will use the notation $v_{xy} := \frac{v_x+v_y}{2}$ for $v\in V$.

\begin{lem}
\label{lem:eqextr}
Assume that $N_k>0$ and $a^2N_{k+1}=aN_k=N_{k-1}$. Then there exists
$v\in V$ so that the following hold for $z,z'\in P\backslash\{x,y\}$.
\begin{enumerate}[leftmargin=*,itemsep=\topsep,label=\alph*.]
\item If $z$ is a maximal element of $P$, then $v_z=1-a$.
\item If $z$ is a minimal element of $P$, then $v_z=0$.
\item If $e_{zz'}$ is $k$-extreme for $z<z'$ with
$z\not\in P_{<x}$ or $z'\not\in P_{>y}$, then $v_z=v_{z'}$.
\item If $e_{zxy}$ is $k$-extreme for $z\in P_{x<\cdot<y}\cup 
P_{<y,\parallel x}$, then $v_z = v_{xy}-\frac{a}{2}$.
\item If $e_{xyz}$ is $k$-extreme for $z\in P_{x<\cdot<y}\cup 
P_{>x,\parallel y}$, then $v_z = v_{xy}+\frac{a}{2}$.
\item If $e_{zxy}$ is $k$-extreme for $z\in P_{<x}$, then
$v_z = v_{xy}+\frac{a}{2}$.
\item If $e_{xyz}$ is $k$-extreme for
$z\in P_{>y}$, then $v_z=v_{xy}-\frac{a}{2}$.
\end{enumerate}
\end{lem}

\begin{proof}
Proposition \ref{prop:geom} states that there exists $v\in V$ so that 
$h_K(u)=ah_L(u)+\langle u,v\rangle$ for every $k$-extreme vector $u$.
Parts (a) and (b) follow from Lemmas \ref{lem:hcoord} and 
\ref{lem:kxcoord}, (c) follows from Lemma \ref{lem:htrans}, and
the remaining parts follow from Lemma \ref{lem:hanchor}.
\end{proof}

Lemma \ref{lem:eqextr} shows that each $k$-extreme direction yields a 
linear constraint on the vector $v$. The basic principle behind our 
results is that if there are too many $k$-extreme vectors, the resulting 
system of linear equations for $v$ has no solution. The latter entails a 
contradiction, as the existence of $v$ is guaranteed by 
Proposition~\ref{prop:geom}. This will enable us to reason that some 
vectors must not be $k$-extreme, which gives rise to the explicit 
combinatorial conditions in our main results (as in 
Definition~\ref{defn:conditions}) using the description of $k$-extreme 
directions in the previous section. How to reason about the existence of 
solutions is far from obvious, however, and will require a careful 
analysis of the structure of the underlying poset.

The fact that many transition vectors, viz.\ those that appear
in parts (a)--(f) of Lemma \ref{lem:kxtrans}, are always $k$-extreme will 
enable us to fix many entries of $v$ at the outset of the analysis. It 
will turn out that the nontrivial structure of the extremals is largely 
(but not entirely) controlled by the value of $v_{xy}$. In particular, let 
us define four properties that will play a central role in the analysis; 
the significance of these properties will become evident in the proofs.

\begin{defn}
\label{defn:condpf}
We define the following properties:
\begin{enumerate}[labelwidth=\widthof{$(\mathscr{M}^*)$},itemsep=\topsep,align=center]
\item[$(\mathscr{M})$] $a\ne \frac{2}{3}(1-v_{xy})$.
\item[$(\mathscr{M}^*)$] $a\ne 2v_{xy}$.
\item[$(\mathscr{E})$] $a\ne -2v_{xy}$.
\item[$(\mathscr{E}^*)$] $a\ne 2(1-v_{xy})$.
\end{enumerate}
\end{defn}

We now briefly outline the steps in the proof of our main results. We 
first show in section \ref{sec:pfe} that $(\mathscr{E}){\Rightarrow}$\Ek\ 
and $(\mathscr{E}^*){\Rightarrow}$\Ekd. In section \ref{sec:pfm}, we 
show that $(\mathscr{M}){\Rightarrow}$\Mk\ and 
$(\mathscr{M}^*){\Rightarrow}$\Mkd. In section \ref{sec:pff} we argue that 
if $a\ne\frac{1}{2}$, then either $(\mathscr{M})$ and $(\mathscr{E})$, or 
$(\mathscr{M}^*)$ and $(\mathscr{E}^*)$, must hold. This simultaneously 
proves both Theorems \ref{thm:maingeom} and \ref{thm:mainflat} (because 
the 
implication (c)$\Rightarrow$(a) of Theorem \ref{thm:mainflat}, which was 
proved in section \ref{sec:easy}, then shows that $a\ne\frac{1}{2}$ 
implies $a=1$). Finally, we prove Theorem \ref{thm:maindouble} in section 
\ref{sec:pfd} using additional arguments specific to this case.

\subsection{The conditions $(\mathscr{E}),(\mathscr{E}^*)$}
\label{sec:pfe}

The aim of this section is to understand the origin of conditions \Ek\ and 
\Ekd\ in Definition \ref{defn:conditions}, which are concerned with 
elements $z\in P_{<x}$ or $P_{>y}$. We begin by computing $v_z$ for these 
elements.

\begin{lem}
\label{lem:vend}
Assume that $N_k>0$ and $a^2N_{k+1}=aN_k=N_{k-1}$, and let $v\in V$ be 
the vector provided by Lemma \ref{lem:eqextr}. Then the following hold.
\begin{enumerate}[leftmargin=*,itemsep=\topsep,label=\alph*.]
\item $v_z=0$ for every $z\in P_{<x}$.
\item $v_z=1-a$ for every $z\in P_{>y}$.
\end{enumerate}
\end{lem}

\begin{proof}
If $z\in P_{<x}$, let $z\gtrdot z_1 \gtrdot z_2\gtrdot \cdots \gtrdot z_r$ 
be a decreasing chain from $z$ to a minimal element $z_r$ of $P$. Then
$z_i\in P_{<x}$ for all $i$, so that
Lemma \ref{lem:kxtrans}(a) and Lemma \ref{lem:eqextr}(c) yield
$v_z=v_{z_1}=v_{z_2}=\cdots=v_{z_r}$. Part (a) follows as $v_{z_r}=0$ by 
Lemma \ref{lem:eqextr}(b). The proof of part (b) follows in a completely 
analogous fashion using Lemma \ref{lem:eqextr}(a).
\end{proof}

We now show how the combinatorial conditions of \Ek\ and \Ekd\ arise.

\begin{lem}
\label{lem:ek}
Assume that $N_k>0$ and $a^2N_{k+1}=aN_k=N_{k-1}$, and let $v\in V$ be
the vector provided by Lemma \ref{lem:eqextr}. Then the following hold.
\begin{enumerate}[leftmargin=*,itemsep=\topsep,label=\alph*.]
\item If $(\mathscr{E})$ holds, then every
$z\in P_{<x}$ satisfies $|P_{z<\cdot<y}\cup\{x\}|>k$.
\item If $(\mathscr{E}^*)$ holds, then every
$z\in P_{>y}$ satisfies $|P_{x<\cdot<z}\cup\{y\}|>k$.
\end{enumerate}
\end{lem}

\begin{proof}
Suppose there exists $z\in P_{<x}$ such that $|P_{z<\cdot<y}\cup\{x\}|\le 
k$. Then the latter holds automatically for any $z'\ge 
z$ (as then $P_{z'<\cdot<y}\subseteq P_{z<\cdot<y}$). In particular, we 
consider $z'\ge z$ that is a maximal element of $P_{<x}$, so that 
$z'\lessdot x$. Then Lemma~\ref{lem:kxanchor}(c), Lemma 
\ref{lem:eqextr}(f), and Lemma \ref{lem:vend}(a) yield 
$0=v_{z'}=v_{xy}+\frac{a}{2}$, which is the converse of $(\mathscr{E})$. 
Thus we proved the contrapositive of part (a).

Similarly, if there exists $z\in P_{>y}$ so that 
$|P_{x<\cdot<z}\cup\{y\}|\le k$, let $y\lessdot z'\le z$. Then Lemma 
\ref{lem:kxanchor}(d), Lemma \ref{lem:eqextr}(g), and Lemma 
\ref{lem:vend}(b) yield $1-a=v_{z'}=v_{xy}-\frac{a}{2}$, which is the 
converse of $(\mathscr{E}^*)$. Thus we proved the contrapositive of part 
(b).
\end{proof}

Lemma \ref{lem:ek} states that $(\mathscr{E})$ implies the first part of 
\Ek, and that $(\mathscr{E}^*)$ implies the first part of \Ekd. The second 
part of \Ek\ and of \Ekd\ is already implied by the first part when 
$P_{<x}$ and $P_{>y}$ are nonempty, so that these additional conditions 
only need to be established in the special case that $x$ is minimal or 
that $y$ is maximal. We do not need to consider these cases separately, 
however, as we can always modify the poset to avoid this situation without 
changing the linear extension numbers $N_k$ by adding a globally minimal 
and maximal element to $P$. We postpone this straightforward argument to 
the proof of Corollary \ref{cor:mainmain}.

\subsection{The conditions $(\mathscr{M}),(\mathscr{M}^*)$}
\label{sec:pfm}

The aim of this section is to understand the origin of conditions \Mk\ and 
\Mkd\ in Definition \ref{defn:conditions}, which are concerned with 
elements $z\in P_{>x,\not\ge y}=P_{x<\cdot<y}\cup P_{>x,\parallel y}$ or 
$P_{<y,\not\le x}=P_{x<\cdot<y}\cup P_{<y,\parallel x}$. The argument in 
the case that $z\in P_{>x,\parallel y}$ or $P_{<y,\parallel x}$ is 
similar to section \ref{sec:pfe}. The case $z\in P_{x<\cdot<y}$ is 
more subtle, however, and will require additional insights.

\subsubsection{The case $z\in P_{>x,\parallel y}$ or $z\in P_{<y,\parallel x}$}

We begin by computing $v_z$.

\begin{lem}
\label{lem:vmid}
Assume that $N_k>0$ and $a^2N_{k+1}=aN_k=N_{k-1}$, and let $v\in V$ be 
the vector provided by Lemma \ref{lem:eqextr}. Then the following hold.
\begin{enumerate}[leftmargin=*,itemsep=\topsep,label=\alph*.]
\item $v_z=0$ for every $z\in P_{<y,\parallel x}$.
\item $v_z=1-a$ for every $z\in P_{>x,\parallel y}$.
\item $v_z=v_{z'}$ for all $z,z'\in P_{x<\cdot<y}$ with $z<z'$.
\end{enumerate}
\end{lem}

\begin{proof}
If $z\in P_{<y,\parallel x}$, let $z\gtrdot z_1\gtrdot 
z_2\gtrdot\cdots\gtrdot z_r$ by a decreasing chain from $z$ to a minimal 
element $z_r$ of $P_{<y,\parallel x}$. If $z_r$ is minimal in $P$, we stop 
the chain at this point. Otherwise, $z_r\gtrdot z_{r+1}$ for some
element $z_{r+1}\in P_{<x}$, and we can continue the chain
$z_{r+1}\gtrdot\cdots\gtrdot z_s$ until we reach a minimal element $z_s$ 
of $P$.

Now note that Lemma \ref{lem:kxtrans}(a,b) and
Lemma \ref{lem:eqextr}(c) yield $v_z=v_{z_1}=\cdots=v_{z_r}$ and
$v_{z_{r+1}}=\cdots=v_{z_s}$. On the other hand, as $z_r$ is a minimal 
element of $P_{<y,\parallel x}$, we have $P_{<z_r}\subseteq P_{<x}$, and 
thus Lemma \ref{lem:kxtrans}(d) and Lemma \ref{lem:eqextr}(c) yield
$z_r=z_{r+1}$. Part (a) follows as $v_{z_s}=0$ by Lemma 
\ref{lem:eqextr}(b).
The proof of part (b) follows in a completely
analogous fashion using Lemma \ref{lem:eqextr}(a).

For part (c), let
$z\lessdot z_1\lessdot \cdots\lessdot z_r\lessdot z'$ be a chain 
connecting $z,z'$, and note that we must have $z_i\in P_{x<\cdot<y}$
for all $i$. Thus Lemma \ref{lem:kxtrans}(c) and Lemma 
\ref{lem:eqextr}(c) imply that $v_z=v_{z_1}=\cdots=v_{z_r}=v_{z'}$, 
concluding the proof.
\end{proof}

We now show how the conditions of \Mk\ and \Mkd\ arise
in the present case.

\begin{lem}
\label{lem:mkxy}
Assume that $N_k>0$ and $a^2N_{k+1}=aN_k=N_{k-1}$, and let $v\in V$ be
the vector provided by Lemma \ref{lem:eqextr}. Then the following hold.
\begin{enumerate}[leftmargin=*,itemsep=\topsep,label=\alph*.]
\item If $(\mathscr{M})$ holds, then every
$z\in P_{>x,\parallel y}$ satisfies $|P_{<z}|+|P_{>y}|>n-k$.
\item If $(\mathscr{M}^*)$ holds, then every
$z\in P_{<y,\parallel x}$ satisfies $|P_{>z}|+|P_{<x}|>n-k$.
\end{enumerate}
\end{lem}

\begin{proof}
Suppose there exists $z\in P_{>x,\parallel y}$ such that
$|P_{<z}|+|P_{>y}|\le n-k$. Then the latter condition holds automatically 
for any $z'\le z$ (as then $P_{<z'}\subseteq P_{<z}$). In particular, 
we consider $z'\le z$ that is a minimal element of $P_{>x,\parallel y}$.

Suppose first that $x\lessdot z'$. Lemma \ref{lem:kxanchor}(b), Lemma 
\ref{lem:eqextr}(e), and Lemma \ref{lem:vmid}(b) then yield 
$1-a=v_{z'}=v_{xy}+\frac{a}{2}$, which is the converse of $(\mathscr{M})$. 
Thus we have proved the contrapositive of part (a) in this case.

On the other hand, if $z'$ does not cover $x$, then we must have 
$P_{x<\cdot<z'}\subseteq P_{x<\cdot<y}$ as $z'$ was chosen to be minimal 
in $P_{>x,\parallel y}$. Moreover, note that $|P_{x<\cdot<y}|\le k-1$ by 
Corollary \ref{cor:eqpos}. Applying Lemma \ref{lem:kxtrans}(g) and
Lemma \ref{lem:eqextr}(c) then shows that $v_{z'}=v_{z''}$ for 
$x<z''\lessdot z'$. In particular, Lemma \ref{lem:vmid}(b,c) yields
$v_{z'''}=1-a$ for all $z'''\in P_{x<\cdot\le z''}$. Now choose $x\lessdot 
z'''\le z''<z$, and note that $|P_{<z'''}|+|P_{>y}|\le n-k$ still holds.
Applying Lemma \ref{lem:kxanchor}(b) and Lemma \ref{lem:eqextr}(e)
yields $1-a=v_{xy}+\frac{a}{2}$, which is the converse of $(\mathscr{M})$.
This completes the proof of part (a).

The proof of part (b) is completely analogous, but now we use
Lemma \ref{lem:kxanchor}(a), Lemma
\ref{lem:eqextr}(d), Lemma \ref{lem:vmid}(a), and
Lemma \ref{lem:kxtrans}(h).
\end{proof}

\subsubsection{The case $z\in P_{x<\cdot<y}$}

Let us begin by explaining the basic difficulty in this case. So far, all 
our arguments started with the construction of a chain that goes from the 
element $z$ of interest to a maximal or minimal element of $P$ without 
passing through $x$ or $y$. We observed that such chains can be 
constructed in such a way that all the transition vectors along the chain 
are $k$-extreme, so that we can compute the value of $v_z$ using Lemma 
\ref{lem:eqextr} (cf.\ Lemmas \ref{lem:vend} and \ref{lem:vmid}). However, 
when $z\in P_{x<\cdot<y}$ it is not even clear that there exists any chain 
connecting $z$ to a minimal or maximal element of $P$ that does not pass 
through $x$ or $y$. In the absence of such a chain we would have no 
mechanism to obtain information about $v_z$.

We presently aim to show that when the Kahn-Saks inequality holds with 
equality, such a chain must always exist. This is not obvious, and arises 
here in a rather subtle manner from the equality conditions.

We begin by showing that the combinatorial conditions of \Mk\ and \Mkd\ 
cannot simultaneously fail to hold for elements of $P_{x<\cdot<y}$.

\begin{lem}
\label{lem:notbothfail}
Assume that $N_k>0$ and $a^2N_{k+1}=aN_k=N_{k-1}$. Then for any comparable 
elements $z,z'\in P_{x<\cdot<y}$, the following hold.
\begin{enumerate}[leftmargin=*,itemsep=\topsep,label=\alph*.]
\item If $|P_{<z}|+|P_{>y}|\le n-k$, then $|P_{>z'}|+|P_{<x}|>n-k$.
\item If $|P_{>z}|+|P_{<x}|\le n-k$, then $|P_{<z'}|+|P_{>y}|>n-k$.
\end{enumerate}
\end{lem}

\begin{proof}
It suffices to prove part (a), as part (b) is the contrapositive of part 
(a) with the roles of $z,z'$ reversed. As $z,z'$ are comparable, there is 
a chain $x\lessdot z_1\lessdot\cdots\lessdot z_r\lessdot y$ so that
$z_i=z$ and $z_j=z'$ for some $i,j$. Now suppose that part (a) fails, that 
is, that $|P_{<z}|+|P_{>y}|\le n-k$ and $|P_{>z'}|+|P_{<x}|\le n-k$.
Then certainly 
$|P_{<z_1}|+|P_{>y}|\le n-k$ and $|P_{>z_r}|+|P_{<x}|\le n-k$ as well.
Applying Lemma \ref{lem:kxanchor}(a,b), Lemma \ref{lem:eqextr}(d,e), and
Lemma \ref{lem:vmid}(c) yields
$v_{xy}+\frac{a}{2}=v_{z_1}=v_{z_r}=v_{xy}-\frac{a}{2}$, which
entails a contradiction as $a>0$. Thus part (a) must hold, concluding the 
proof.
\end{proof}

We now use Lemma \ref{lem:notbothfail} to reason that if the combinatorial 
condition of \Mk\ or of \Mkd\ fails, then there must exist a chain 
starting from any point in $P_{x<\cdot<y}$ that leaves this set without 
passing through $x$ or $y$.

\begin{lem}
\label{lem:escape}
Assume that $N_k>0$ and $a^2N_{k+1}=aN_k=N_{k-1}$, and let $z\in 
P_{x<\cdot<y}$.
\begin{enumerate}[leftmargin=*,itemsep=\topsep,label=\alph*.]
\item If $|P_{<z}|+|P_{>y}|\le n-k$, then any
$z\le z_1\lessdot y$ satisfies $z_1\lessdot z_2$ for some $z_2\in 
P_{>x,\parallel y}$.
\item If $|P_{>z}|+|P_{<x}|\le n-k$, then any
$z\ge z_1\gtrdot x$ satisfies $z_1\gtrdot z_2$ for some
$z_2\in P_{<y,\parallel x}$.
\end{enumerate}
\end{lem}

\begin{proof}
Suppose that $|P_{<z}|+|P_{>y}|\le n-k$, and let
$z\le z_1\lessdot y$. Suppose that (a) fails, that is, that
$z_1$ is not covered by any element of $P_{>x,\parallel y}$. Then we must 
have $P_{>z_1}=P_{\ge y}$. We can therefore estimate using
Lemma \ref{lem:notbothfail}(a) and Corollary \ref{cor:eqpos}
$$
	n-k < |P_{>z_1}|+|P_{<x}| =
	|P_{<x}|+|P_{>y}|+1 < n-k,
$$
which entails a contradiction. Thus part (a) is proved, and the proof of 
part (b) is completely analogous using Lemma \ref{lem:notbothfail}(b).
\end{proof}

With this structural information in hand, we can proceed to proving a 
counterpart of 
Lemma \ref{lem:mkxy} for elements $z\in P_{x<\cdot<y}$.

\begin{lem}
\label{lem:mkmid}
Assume that $N_k>0$ and $a^2N_{k+1}=aN_k=N_{k-1}$, and let $v\in V$ be
the vector provided by Lemma \ref{lem:eqextr}. Then the following hold.
\begin{enumerate}[leftmargin=*,itemsep=\topsep,label=\alph*.]
\item If $(\mathscr{M})$ holds, then every
$z\in P_{x<\cdot<y}$ satisfies $|P_{<z}|+|P_{>y}|>n-k$.
\item If $(\mathscr{M}^*)$ holds, then every
$z\in P_{x<\cdot<y}$ satisfies $|P_{>z}|+|P_{<x}|>n-k$.
\end{enumerate}
\end{lem}

\begin{proof}
Let $z\in P_{x<\cdot<y}$ and
$x\lessdot z_1\le z\le z_2\lessdot y$. Suppose that
$|P_{<z}|+|P_{>y}|\le n-k$. Then Lemma \ref{lem:escape}(a) shows that 
$z_2\lessdot z_3$ for some $z_3\in P_{>x,\parallel y}$, and Lemma 
\ref{lem:notbothfail}(a) yields $|P_{>z_2}|+|P_{<x}|>n-k$. Moreover, 
note that the sets $P_{x<\cdot<y}\cup P_{x<\cdot<z_3}$, $P_{<x}$, and
$P_{>z_2}$ are disjoint. We can therefore estimate
$$
	|P_{x<\cdot<y}\cup P_{x<\cdot<z_3}| +
	n-k+1 \le |P_{x<\cdot<y}\cup P_{x<\cdot<z_3}|
	+ |P_{>z_2}|+|P_{<x}|\le n,
$$
which yields $|P_{x<\cdot<y}\cup P_{x<\cdot<z_3}| \le k-1$.
Thus Lemma \ref{lem:kxtrans}(g), Lemma \ref{lem:eqextr}(c), and
Lemma \ref{lem:vmid}(b,c) show that
$v_{z_1}=v_{z_2}=v_{z_3}=1-a$. On the other hand, as $|P_{<z}|+|P_{>y}|\le 
n-k$ clearly implies $|P_{<z_1}|+|P_{>y}|\le n-k$, Lemma 
\ref{lem:kxanchor}(b) and Lemma \ref{lem:eqextr}(e) yield
$1-a=v_{z_1}=v_{xy}+\frac{a}{2}$ which is the converse of 
$(\mathscr{M})$. 

Thus we have proved the contrapositive of part (a).
The proof of part (b) is completely analogous, but now we use
Lemma~\ref{lem:escape}(b),
Lemma~\ref{lem:notbothfail}(b),
Lemma~\ref{lem:kxtrans}(h), 
Lemma~\ref{lem:vmid}(a,c),
Lemma~\ref{lem:kxanchor}(a), and Lemma~\ref{lem:eqextr}(d).
\end{proof}

Combining Lemma \ref{lem:mkxy} and \ref{lem:mkmid}, we conclude that
$(\mathscr{M}){\Rightarrow}$\Mk\ and $(\mathscr{M}^*){\Rightarrow}$\Mkd.

\subsection{The case $a\ne\frac{1}{2}$}
\label{sec:pff}

We are now ready to prove Theorems \ref{thm:maingeom} and 
\ref{thm:mainflat}. The basic observation is the following simple fact.

\begin{lem}
\label{lem:nothalf}
If $a\ne\frac{1}{2}$, then either $(\mathscr{M})$ and $(\mathscr{E})$ 
hold, or $(\mathscr{M}^*)$ and $(\mathscr{E}^*)$ hold.
\end{lem}

\begin{proof} 
If the conclusion fails, then it must be the case that 
either $(\mathscr{M})$ and $(\mathscr{M}^*)$ fail, or $(\mathscr{M})$ and 
$(\mathscr{E}^*)$ fail, or $(\mathscr{E})$ and $(\mathscr{M}^*)$ fail, or 
$(\mathscr{E})$ and $(\mathscr{E}^*)$ fail. We now show that each of these 
possibilities yields a contradiction.
\begin{enumerate}[leftmargin=*,itemsep=\topsep,label=\alph*.]
\item
If $(\mathscr{M})$ and $(\mathscr{M}^*)$ both fail, then
$
	a=\frac{2}{3}(1-v_{xy})=2v_{xy}
$
implies that $a=\frac{1}{2}$, which contradicts the assumption.
\item
If $(\mathscr{M})$ and $(\mathscr{E}^*)$ both fail, then
$
	a=\frac{2}{3}(1-v_{xy})=2(1-v_{xy})
$
implies that $a=0$, which is impossible as $a>0$ by assumption.
\item
If $(\mathscr{E})$ and $(\mathscr{M}^*)$ both fail, then
$
	a=2v_{xy} = -2v_{xy}
$
implies that $a=0$, which is impossible as $a>0$ by assumption.
\item
If $(\mathscr{E})$ and $(\mathscr{E}^*)$ both fail, then
$
	-2v_{xy}=2(1-v_{xy})
$
is evidently impossible.
\end{enumerate} 
This concludes the proof.
\end{proof}

We can now conclude the following.

\begin{cor}
\label{cor:mainmain}
Assume that $N_k>0$ and $a^2N_{k+1}=aN_k=N_{k-1}$ with $a\ne\frac{1}{2}$.
Then either \Mk\ and \Ek\ hold, or \Mkd\ and \Ekd\ hold.
\end{cor}

\begin{proof}
If $P_{<x}\ne\varnothing$ and $P_{>y}\ne\varnothing$, the second condition 
of \Ek\ and \Ekd\ is subsumed by the first. Then the result follows 
from Lemmas \ref{lem:nothalf}, \ref{lem:ek}, \ref{lem:mkxy}, and 
\ref{lem:mkmid}.

Otherwise, we augment the poset $P$ by adding a globally minimal and 
maximal element, i.e., $\hat P:=P\cup\{\hat 0,\hat 1\}$ with the 
additional relations $\hat 0<z<\hat 1$ for all $z\in P$. As $\hat 0$ and 
$\hat 1$ must appear at the beginning and end of every linear extension, 
the numbers $N_k$ are unchanged if we replace $P$ by 
$\hat P$. Thus we conclude that either \Mk\ and \Ek\ hold 
for $\hat P$, or \Mkd\ and \Ekd\ hold for $\hat P$. It remains to verify 
that these conditions for $\hat P$ imply the corresponding conditions for 
$P$.

To this end, note that \Ek\ for $\hat P$ states that $|\hat 
P_{z<\cdot<y}\cup\{x\}|>k$ for all $z\in \hat P_{<x}$. Applying this 
condition to $z\in P_{<x}$ yields the first part of \Ek\ for $P$, while 
applying this condition to $z=\hat 0$ yields the second part of \Ek\ for 
$P$. The proof that \Ekd\ for $\hat P$ implies \Ekd\ for $P$ is completely 
analogous. On the other hand, note that \Mk\ for $\hat P$ states that 
$|\hat P_{<z}|+|\hat P_{>y}|>n+2-k$ for all $z\in\hat P_{>x,\not\ge y}$
(as $|\hat P|=n+2$). As $\hat P_{<z}=P_{<z}\cup\{\hat 0\}$ and
$\hat P_{>y}=P_{>y}\cup\{\hat 1\}$, the validity of \Mk\ for $P$ follows 
readily. The proof that \Mkd\ for $\hat P$ implies \Mkd\ for $P$ is 
completely analogous.
\end{proof}

We now complete the proofs of Theorems \ref{thm:maingeom} and
\ref{thm:mainflat}.

\begin{proof}[Proof of Theorem \ref{thm:mainflat}]
The implications (c)$\Rightarrow$(b)$\Rightarrow$(a) were proved in 
section  \ref{sec:easy}. The implication (a)$\Rightarrow$(c) follows by 
applying Corollary \ref{cor:mainmain} with $a=1$.
\end{proof}

\begin{proof}[Proof of Theorem \ref{thm:maingeom}]
The implication (b)$\Rightarrow$(a) is trivial. Conversely, suppose that
(a) holds. Then we clearly have $a^2N_{k+1}=aN_k=N_{k-1}$ for some $a>0$ 
(as we assumed $N_k>0$). If $a=\frac{1}{2}$, we have 
$N_{k+1}=2N_k=4N_{k-1}$. If $a\ne\frac{1}{2}$, then 
Corollary~\ref{cor:mainmain} and the implication 
(c)$\Rightarrow$(a) of Theorem \ref{thm:mainflat} yield 
${N_{k+1}=N_k=N_{k-1}}$. Thus we have proved (a)$\Rightarrow$(b),
concluding the proof.
\end{proof}

\subsection{The case $a=\frac{1}{2}$}
\label{sec:pfd}

The proof of Theorem \ref{thm:maindouble}, which is concerned with the 
equality case $a^2N_{k+1}=aN_k=N_{k-1}$ for $a=\frac{1}{2}$ (i.e., 
$N_{k+1}=2N_k=4N_{k-1}$), requires us to obtain additional information on 
the structure of $P$. Let us begin by explaining why we must have 
$P_{\parallel x,\parallel y}=\varnothing$ in this case.

\begin{lem}
\label{lem:incomparable}
Let $N_k>0$ and $a^2N_{k+1}=aN_k=N_{k-1}$.
If $P_{\parallel x,\parallel y}\ne\varnothing$, then $a=1$.
\end{lem}

\begin{proof}
If $P_{\parallel x,\parallel y}\ne\varnothing$, there is a chain
$z_0\lessdot z_1\lessdot \cdots\lessdot z_r$ from a minimal element $z_0$ 
to a maximal element $z_r$ of $P_{\parallel x,\parallel y}$. Lemma 
\ref{lem:kxtrans}(f) and Lemma \ref{lem:eqextr}(c) yield
$v_{z_0}=\cdots=v_{z_r}$. To conclude the proof, we show that
$v_{z_0}=0$ and $v_{z_r}=1-a$.

If $z_r$ is maximal in $P$, then $v_{z_r}=1-a$ by 
Lemma~\ref{lem:eqextr}(a). Otherwise, if there exists $z_r\lessdot z\in 
P_{\parallel y}$, we must have $z\in P_{>x,\parallel y}$ (as $z\in 
P_{\parallel x}$ would contradict maximality of $z_r$ in $P_{\parallel 
x,\parallel y}$, while $z_r\lessdot z\in P_{\le x}$ would contradict 
$z_r\in P_{\parallel x}$). Then Lemma~\ref{lem:kxtrans}(f), 
Lemma~\ref{lem:eqextr}(c), and Lemma~\ref{lem:vmid}(b) yield 
$v_{z_r}=v_z=1-a$. Finally, if $z_r$ is not maximal in $P$ and there does 
not exist $z_r\lessdot z\in P_{\parallel y}$, we must have 
$P_{>z_r}\subseteq P_{>y}$ (as $z_r< z\in P_{\le y}$ would contradict 
$z_r\in P_{\parallel y}$). Then Lemma~\ref{lem:kxtrans}(e), 
Lemma~\ref{lem:eqextr}(c), and Lemma~\ref{lem:vend}(b) yield 
$v_{z_r}=1-a$. The proof that $v_{z_0}=0$ is completely analogous. 
\end{proof}

Next, we prove a result that will be used to show that 
$P_{x<\cdot<y}=\varnothing$ when $a=\frac{1}{2}$. We state a slightly more 
general form than is needed in the proof of Theorem~\ref{thm:maindouble}, 
as it will provide some additional information (Lemma \ref{lem:mutex} 
below): in essence, we show that \Mk\ and \Ekd\ (and analogously \Mkd\ and 
\Ek) cannot both hold.

\begin{lem}
\label{lem:midnono}
The following hold.
\begin{enumerate}[leftmargin=*,itemsep=\topsep,label=\alph*.]
\item Let $x\lessdot z\in P_{>x,\not\ge y}$ and 
$y\lessdot z'$. If $|P_{<z}|+|P_{>y}|>n-k$, then
$|P_{x<\cdot<z'}\cup\{y\}|\le k$.
\item Let $y\gtrdot z\in P_{<y,\not\le x}$ and 
$z'\lessdot x$. If $|P_{>z}|+|P_{<x}|>n-k$, then
$|P_{z'<\cdot<y}\cup\{x\}|\le k$.
\end{enumerate}
\end{lem}

\begin{proof}
Suppose that $x\lessdot z\in P_{>x,\not\ge y}$ satisfies
$|P_{<z}|+|P_{>y}|>n-k$, and let $y\lessdot z'$.
As the sets $P_{>x,\not>y}$, $P_{<z}$, and $P_{>y}$
are disjoint, and as $P_{x<\cdot<z'}\subseteq P_{>x,\not> y}$, we have
$$
	|P_{x<\cdot<z'}| +
	n-k+1
	\le 
	|P_{>x,\not> y}| + |P_{<z}| + |P_{>y}| \le n,
$$
so that $|P_{x<\cdot<z'}|\le  k-1$. Thus $|P_{x<\cdot<z'}\cup\{y\}|\le k$, 
which completes the proof of part (a). The proof of part (b) is completely 
analogous.
\end{proof}

We can now prove Theorem \ref{thm:maindouble}.

\begin{proof}[Proof of Theorem \ref{thm:maindouble}]
The implications (c)$\Rightarrow$(b)$\Rightarrow$(a) were proved in 
section \ref{sec:easy}, so it remains to prove the implication 
(a)$\Rightarrow$(c). By the same augmentation argument as in the proof of 
Corollary \ref{cor:mainmain}, we can assume without loss of generality
in the remainder of the proof that $P_{<x}\ne\varnothing$ and 
$P_{>y}\ne\varnothing$.

By assumption, we have $N_k>0$ and $a^2N_{k+1}=aN_k=N_{k-1}$ with 
$a=\frac{1}{2}$. Thus the conditions of Definition \ref{defn:condpf} 
reduce to
$$
	(\mathscr{M})\Leftrightarrow (\mathscr{M}^*)\Leftrightarrow
	v_{xy}\ne \tfrac{1}{4},\qquad
	(\mathscr{E})\Leftrightarrow v_{xy}\ne -\tfrac{1}{4},\qquad
	(\mathscr{E}^*) \Leftrightarrow v_{xy}\ne\tfrac{3}{4}.
$$
If $v_{xy}\ne\frac{1}{4}$, then $(\mathscr{M})$, $(\mathscr{M}^*)$, and 
either $(\mathscr{E})$ or $(\mathscr{E}^*)$ hold. Then Lemmas 
\ref{lem:ek}, \ref{lem:mkxy}, and \ref{lem:mkmid} and the implication 
(c)$\Rightarrow$(a) of Theorem \ref{thm:mainflat}
yield $N_{k+1}=N_k=N_{k-1}$, which contradicts the 
assumption. Thus we must have $v_{xy}=\frac{1}{4}$. We therefore conclude
that both $(\mathscr{E})$ and $(\mathscr{E}^*)$ hold, which implies
using Lemma \ref{lem:ek} (and as we assumed that $P_{<x},P_{>y}$ are 
nonempty) that \Ek\ and \Ekd\ both hold.

That $P_{\parallel x,\parallel y}=\varnothing$ was shown in Lemma 
\ref{lem:incomparable}. We claim that also $P_{x<\cdot<y}=\varnothing$. 
Indeed, if the latter does not hold, then there exist $x\lessdot z_1\le 
z_2\lessdot y$, while there exist $z'\lessdot x$ and $y\lessdot z''$ as we 
assumed that $P_{<x},P_{>y}$ are nonempty. By Lemma \ref{lem:notbothfail}, 
we have either $|P_{<z_1}|+|P_{>y}|>n-k$ or $|P_{>z_2}|+|P_{<x}|>n-k$. 
Thus Lemma \ref{lem:midnono} shows that either 
$|P_{z'<\cdot<y}\cup\{x\}|\le k$ or $|P_{x<\cdot<z''}\cup\{y\}|\le k$. But 
this contradicts the validity of \Ek\ and \Ekd, establishing the claim.

It remains to prove that \Ck\ holds. Suppose to the contrary that there 
exist $z\in P_{<y,\parallel x}$, $z'\in P_{>x,\parallel y}$ with $z<z'$ so 
that $|P_{z<\cdot<y}|+|P_{x<\cdot<z'}|\le k-2$.
As $P_{\parallel x,\parallel y}=P_{x<\cdot<y}=\varnothing$, we must have
$P_{z<\cdot<z'}\subset P_{<y,\parallel x}\cup P_{>x,\parallel y}$.
Thus there must exist $z\le z_1\lessdot z_2\le z'$ so that
$z_1\in P_{<y,\parallel x}$ and $z_2\in P_{>x,\parallel y}$, and
$$
	|P_{z_1<\cdot<y}|+|P_{x<\cdot<z_2}|\le
	|P_{z<\cdot<y}|+|P_{x<\cdot<z'}|\le k-2.
$$
Consequently $e_{z_1z_2}$ is $k$-extreme by Lemma \ref{lem:kxtrans}(i). 
But then Lemma \ref{lem:eqextr}(c) and Lemma \ref{lem:vmid}(a,b) yield
$0=v_{z_1}=v_{z_2}=1-a$, which contradicts the assumption that 
$a=\frac{1}{2}$. Thus \Ck\ must hold, concluding the proof.
\end{proof}

We conclude this section with an additional fact that is not needed 
in the proofs of our main results, but helps clarify the 
conditions of Theorems \ref{thm:mainflat} and \ref{thm:maindouble}.

\begin{lem}
\label{lem:mutex}
The following hold.
\begin{enumerate}[leftmargin=*,itemsep=\topsep,label=\alph*.]
\item If \Mk\ holds, then \Ekd\ must fail.
\item If \Mkd\ holds, then \Ek\ must fail.
\end{enumerate}
\end{lem}

\begin{proof}
Suppose that \Mk\ holds. We consider four cases.
\begin{enumerate}[leftmargin=*,itemsep=\topsep,label=\textbullet]
\item
If $P_{>x,\not\ge y}\ne\varnothing$ and
$P_{>y}\ne\varnothing$, then \Ekd\ must fail by Lemma 
\ref{lem:midnono}(a).
\item
If $P_{>x,\not\ge y}\ne\varnothing$ and
$P_{>y}=\varnothing$, then \Mk\ implies $|P_{<z}|>n-k$ for $x\lessdot z\in 
P_{>x,\not\ge y}$. As $P_{<z}$ and $P_{>x}$ are disjoint, 
$n-k+|P_{>x}|<|P_{<z}|+|P_{>x}|\le n$ contradicts \Ekd. 
\item If $P_{>x,\not\ge y}=\varnothing$ and $P_{>y}\ne\varnothing$, then
$|P_{x<\cdot<z}\cup\{y\}|=1$ for $y\lessdot z$ contradicts \Ekd.
\item Finally, if
$P_{>x,\not\ge y}=P_{>y}=\varnothing$, then
$|P_{>x}\cup\{y\}|=1$ which contradicts \Ekd. 
\end{enumerate}
This proves part (a). The 
proof of part (b) is completely analogous.
\end{proof}

\subsection{An explicit example}
\label{sec:weird}

A surprising aspect of the results of this paper is that the equality 
cases of the Alexandrov-Fenchel inequality that arise here need not 
respect the lattice structure of the underlying polytopes, as was 
discussed in section \ref{sec:transscale}. The following simple example
illustrates this phenomenon.

Consider the poset $P$ with $|P|=6$ defined by the relations
$$
	z_1\lessdot z_2\lessdot y,\qquad
	x\lessdot z_3\lessdot z_4.
$$
It is readily verified that Theorem \ref{thm:maindouble} yields a doubling 
progression for $k=2$; in fact, we manually compute $N_1=1$, $N_2=2$, 
$N_3=4$.

In this example, the polytopes of Kahn and Saks are given by
\begin{align*}
	K &= \{t\in\mathbb{R}^P:
	0\le t_{z_1}\le t_{z_2}\le t_y=t_x\le t_{z_3}\le t_{z_4}\le 1\},\\
	L &= \{t\in\mathbb{R}^P:
	0\le t_{z_1}\le t_{z_2}\le t_y=1,~
	0=t_x\le t_{z_3}\le t_{z_4}\le 1\}.
\end{align*}
Now note that Proposition \ref{prop:geom} applies with $a=\frac{1}{2}$, so 
that $h_K(u) = h_{aL+v}(u)$ for all $k$-extreme vectors $u$. We claim that 
the translation vector $v$ may be chosen as 
$$
	v_{z_1}=v_{z_2}=v_y=0,\qquad
	v_{z_3}=v_{z_4}=v_x=\tfrac{1}{2}.
$$
This can be read off from the proof of our main results.
Indeed, the values of $v_{z_i}$ are given by Lemma \ref{lem:vmid}.
On the other hand, the proof of Theorem \ref{thm:maindouble} shows that
$v_{xy}=\frac{1}{4}$, so it is compatible with Proposition \ref{prop:geom}  
to choose $v_y=0$ and $v_x=\frac{1}{2}$ (this choice yields $v\not\in V$, 
but this is irrelevant by Remark \ref{rem:dontworry}; the present choice 
was made to make the equality condition easiest to visualize).

From these computations, it is readily seen that
\begin{align*}
	aL+v &=
	\{t\in\mathbb{R}^P:
	0\le t_{z_1}\le t_{z_2}\le t_y=\tfrac{1}{2}=
	t_x\le t_{z_3}\le t_{z_4}\le 1\} \\ &=
	K \cap \{t\in\mathbb{R}^P:t_x=t_y=\tfrac{1}{2}\}.
\end{align*}
Even though the polytopes $K$ and $aL+v$ do not coincide, the proof of our 
main results shows that they must have the same supporting hyperplanes in 
all $k$-extreme directions. On the other hand, note that $aL+v$ is 
obtained by intersecting $K$ by a non-lattice hyperplane, so that the 
equality condition of the Alexandrov-Fenchel inequality does not respect 
the lattice structure of the polytopes $K,L$.

\subsection*{Acknowledgments}

R.v.H.\ was supported in part by NSF grants DMS-1811735, DMS-2054565, 
and DMS-2347954. 
This work was done while A.Y.\ and X.Z.\ were at Princeton University, 
where A.Y.\ was supported in part by summer research funding from the 
Department of Mathematics. The authors thank Swee Hong Chan, Igor Pak, 
Greta Panova, and Yair Shenfeld for many enlightening conversations, and 
the anonymous referees for their helpful comments and suggestions.

\bibliographystyle{abbrv}
\bibliography{ref}

\begin{thebibliography}{10}

\bibitem{AHK17}
K.~Adiprasito, J.~Huh, and E.~Katz.
\newblock Hodge theory of matroids.
\newblock {\em Notices Amer. Math. Soc.}, 64(1):26--30, 2017.

\bibitem{Ale37}
A.~D. Alexandrov.
\newblock Zur {T}heorie der gemischten {V}olumina von konvexen {K}\"orpern
  {II}.
\newblock {\em Mat. Sbornik N.S.}, 2:1205--1238, 1937.

\bibitem{Huh20}
P.~Br\"{a}nd\'{e}n and J.~Huh.
\newblock Lorentzian polynomials.
\newblock {\em Ann. of Math. (2)}, 192(3):821--891, 2020.

\bibitem{BW91}
G.~Brightwell and P.~Winkler.
\newblock Counting linear extensions.
\newblock {\em Order}, 8(3):225--242, 1991.

\bibitem{CP22}
S.~H. Chan and I.~Pak.
\newblock Introduction to the combinatorial atlas.
\newblock {\em Expo. Math.}, 40(4):1014--1048, 2022.

\bibitem{CP24}
S.~H. Chan and I.~Pak.
\newblock Equality cases of the {A}lexandrov-{F}enchel inequality are not in
  the polynomial hierarchy, 2023.
\newblock Preprint arxiv:2309.05764.

\bibitem{CP21}
S.~H. Chan and I.~Pak.
\newblock Log-concave poset inequalities.
\newblock {\em J. Assoc. Math. Res.}, 2(1):53–153, 2024.

\bibitem{CPP23}
S.~H. Chan, I.~Pak, and G.~Panova.
\newblock Extensions of the {K}ahn-{S}aks inequality for posets of width two.
\newblock {\em Comb. Theory}, 3(1):Paper No. 8, 35, 2023.

\bibitem{Ful93}
W.~Fulton.
\newblock {\em Introduction to toric varieties}, volume 131 of {\em Annals of
  Mathematics Studies}.
\newblock Princeton University Press, Princeton, NJ, 1993.

\bibitem{Gro90}
M.~Gromov.
\newblock Convex sets and {K}\"{a}hler manifolds.
\newblock In {\em Advances in differential geometry and topology}, pages 1--38.
  World Sci. Publ., Teaneck, NJ, 1990.

\bibitem{HX23}
J.~Hu and J.~Xiao.
\newblock Numerical characterization of the hard {L}efschetz classes of
  dimension two, 2023.
\newblock Preprint arxiv:2309.05008.

\bibitem{Huh18}
J.~Huh.
\newblock Combinatorial applications of the {H}odge-{R}iemann relations.
\newblock In {\em Proceedings of the {I}nternational {C}ongress of
  {M}athematicians---{R}io de {J}aneiro 2018. {V}ol. {IV}. {I}nvited lectures},
  pages 3093--3111. World Sci. Publ., Hackensack, NJ, 2018.

\bibitem{KS84}
J.~Kahn and M.~Saks.
\newblock Balancing poset extensions.
\newblock {\em Order}, 1(2):113--126, 1984.

\bibitem{MS23}
Z.~Y. Ma and Y.~Shenfeld.
\newblock The extremals of {S}tanley's inequalities for partially ordered sets.
\newblock {\em Adv. Math.}, 436:Paper No. 109404, 72, 2024.

\bibitem{Sch85}
R.~Schneider.
\newblock On the {A}leksandrov-{F}enchel inequality.
\newblock In {\em Discrete geometry and convexity ({N}ew {Y}ork, 1982)}, volume
  440 of {\em Ann. New York Acad. Sci.}, pages 132--141. New York Acad. Sci.,
  New York, 1985.

\bibitem{Sch14}
R.~Schneider.
\newblock {\em Convex bodies: the {B}runn-{M}inkowski theory}.
\newblock Cambridge University Press, Cambridge, 2014.

\bibitem{SvH23}
Y.~Shenfeld and R.~van Handel.
\newblock The extremals of the {A}lexandrov-{F}enchel inequality for convex
  polytopes.
\newblock {\em Acta Math.}, 231(1):89--204, 2023.

\bibitem{Sta81}
R.~P. Stanley.
\newblock Two combinatorial applications of the {A}leksandrov-{F}enchel
  inequalities.
\newblock {\em J. Combin. Theory Ser. A}, 31(1):56--65, 1981.

\bibitem{Sta86}
R.~P. Stanley.
\newblock Two poset polytopes.
\newblock {\em Discrete Comput. Geom.}, 1(1):9--23, 1986.

\bibitem{Sta89}
R.~P. Stanley.
\newblock Log-concave and unimodal sequences in algebra, combinatorics, and
  geometry.
\newblock In {\em Graph theory and its applications: {E}ast and {W}est
  ({J}inan, 1986)}, volume 576 of {\em Ann. New York Acad. Sci.}, pages
  500--535. New York Acad. Sci., New York, 1989.

\bibitem{Tei82}
B.~Teissier.
\newblock Bonnesen-type inequalities in algebraic geometry. {I}. {I}ntroduction
  to the problem.
\newblock In {\em Seminar on {D}ifferential {G}eometry}, volume 102 of {\em
  Ann. of Math. Stud.}, pages 85--105. Princeton Univ. Press, Princeton, N.J.,
  1982.

\end{thebibliography}

\end{document}